\documentclass[a4paper,12pt]{amsart}

\usepackage{amssymb}
\usepackage{latexsym}
\usepackage{amsmath}
\usepackage{amscd}
\usepackage{graphicx}
\usepackage{palatino}

\title[symplectic invariant Lie subalgebras
of symplectic derivation Lie algebras]
{Structure of symplectic invariant Lie subalgebras 
of symplectic derivation Lie algebras}

\author{Shigeyuki Morita}
\address{Graduate School of Mathematical Sciences, 
The University of Tokyo, 
3-8-1 Komaba, 
Meguro-ku, Tokyo, 153-8914, Japan}
\email{morita@ms.u-tokyo.ac.jp}
\author{Takuya Sakasai}
\address{Graduate School of Mathematical Sciences, 
The University of Tokyo, 
3-8-1 Komaba, 
Meguro-ku, Tokyo, 153-8914, Japan}
\email{sakasai@ms.u-tokyo.ac.jp}
\author{Masaaki Suzuki}
\address{Department of Frontier Media Science, 
Meiji University, 
4-21-1 Nakano, Nakano-ku, Tokyo, 164-8525, Japan}
\email{macky@fms.meiji.ac.jp}

\subjclass[2000]{Primary~17B40, Secondary~17B65}
\keywords{derivation Lie algebra, free Lie algebra, symplectic representation, Young diagram, Johnson homomorphism}

\newtheorem{thm}{Theorem}[section]
\newtheorem{prop}[thm]{Proposition}
\newtheorem{lem}[thm]{Lemma}
\newtheorem{cor}[thm]{Corollary}
\theoremstyle{definition}
\newtheorem{definition}[thm]{Definition}
\newtheorem{example}[thm]{Example}
\newtheorem{remark}[thm]{Remark}

\begin{document}

\newcommand{\Mg}{\mathcal{M}_g}
\newcommand{\Mgp}{\mathcal{M}_{g,\ast}}
\newcommand{\Mgb}{\mathcal{M}_{g,1}}

\newcommand{\hg}{\mathfrak{h}_{g,1}}
\newcommand{\ag}{\mathfrak{a}_g}
\newcommand{\Ln}{\mathcal{L}_n}

\newcommand{\Sg}{\Sigma_g}
\newcommand{\Sgb}{\Sigma_{g,1}}
\newcommand{\la}{\lambda}

\newcommand{\Symp}[1]{Sp(2g,\mathbb{#1})}
\newcommand{\symp}[1]{\mathfrak{sp}(2g,\mathbb{#1})}
\newcommand{\gl}[1]{\mathfrak{gl}(n,\mathbb{#1})}

\newcommand{\At}[1]{\mathcal{A}_{#1}^t (H)}
\newcommand{\Hq}{H_{\mathbb{Q}}}

\newcommand{\Ker}{\mathop{\mathrm{Ker}}\nolimits}
\newcommand{\Hom}{\mathop{\mathrm{Hom}}\nolimits}
\renewcommand{\Im}{\mathop{\mathrm{Im}}\nolimits}

\newcommand{\Der}{\mathop{\mathrm{Der}}\nolimits}
\newcommand{\Out}{\mathop{\mathrm{Out}}\nolimits}
\newcommand{\Aut}{\mathop{\mathrm{Aut}}\nolimits}
\newcommand{\Q}{\mathbb{Q}}
\newcommand{\Z}{\mathbb{Z}}
\newcommand{\R}{\mathbb{R}}

\begin{abstract}
We study the structure of the symplectic invariant part
$\mathfrak{h}_{g,1}^{\mathrm{Sp}}$ of the Lie algebra
$\mathfrak{h}_{g,1}$ consisting of symplectic derivations
of the free Lie algebra 
generated by the
rational homology group 
of a closed oriented surface $\Sigma_{g}$ of genus $g$.

First we describe the orthogonal direct sum decomposition of 
this space which is induced by the canonical metric on it and 
compute it explicitly up to degree $20$. In this framework, we give a general constraint
which is imposed on the $\mathrm{Sp}$-invariant component of
the bracket of two elements in $\mathfrak{h}_{g,1}$.
Second we clarify the relations among 
$\mathfrak{h}_{g,1}$ and the other two related Lie algebras
$\mathfrak{h}_{g,*}$ and $\mathfrak{h}_{g}$ which correspond to
the cases of a closed surface $\Sigma_g$ with and without base point $*\in\Sigma_g$.
In particular, based on a theorem of Labute, we formulate a method of determining these differences  
and describe them explicitly up to degree $20$.
Third, by giving a general method of constructing elements of $\mathfrak{h}_{g,1}^{\mathrm{Sp}}$,
we reveal a considerable difference between two particular submodules of it, one is the $\mathrm{Sp}$-invariant part 
of a certain ideal $\mathfrak{j}_{g,1}$
and the other is that
of the Johnson image.

Finally we combine these results to determine the structure of $\mathfrak{h}_{g,1}$
completely up to degree $6$ including the unstable cases
where the genus $1$ case has an independent meaning.
In particular, we see a glimpse of the Galois obstructions 
explicitly from our point of view.

\end{abstract}

\renewcommand\baselinestretch{1.1}
\setlength{\baselineskip}{16pt}

\newcounter{fig}
\setcounter{fig}{0}

\maketitle

\section{Introduction and statements of the main results}\label{sec:intro}

Let $\Sigma_{g,1}$ be a compact oriented surface of genus $g\geq 1$ 
with one boundary component and we denote its first integral homology group
$H_1(\Sigma_{g,1};\Z)$ simply by $H$ and let $H_\Q=H\otimes\Q$. 
We denote by $\mathcal{L}_{g,1}$ the 
free graded Lie algebra generated by $H_\Q$ and let $\mathfrak{h}_{g,1}$
be the graded Lie algebra consisting of {\it symplectic} derivations of $\mathcal{L}_{g,1}$.
Let $\mathfrak{h}^+_{g,1}$ be the ideal consisting of derivations 
with {\it positive} degrees.
This Lie algebra was introduced in the theory of Johnson homomorphisms 
(see \cite{morita89}) and has been investigated extensively. 
We also consider closely related Lie algebras,
denoted by $\mathfrak{h}_{g,*}$ and  $\mathfrak{h}_g$
which correspond to the cases of a closed surface $\Sigma_g$
with and without base point $*\in\Sigma_g$. 

Let $\mathrm{Sp}(2g,\Q)$ be the symplectic group which
we sometimes denote simply by $\mathrm{Sp}$. If we fix a symplectic 
basis of $H_\Q$, then it can be 
considered as the standard representation of $\mathrm{Sp}(2g,\Q)$.
Each piece $\mathfrak{h}_{g,1}(k), \mathfrak{h}_{g,*}(k),\mathfrak{h}_{g}(k)$,
of the three graded Lie algebras,
is naturally an $\mathrm{Sp}$-module so that it has an irreducible decomposition. 
Let $\mathfrak{h}^{\mathrm{Sp}}_{g,1}$ denote the Lie subalgebra of 
$\mathfrak{h}_{g,1}$ consisting of $\mathrm{Sp}$-invariant
elements. We denote by $\mathfrak{h}_{g,1}(2k)^\mathrm{Sp}$ the 
degree $2k$ part of this Lie subalgebra. We use similar notations for the
other two cases $ \mathfrak{h}_{g,*}$ and $\mathfrak{h}_{g}$.

In \cite{morita13} a canonical metric on $(H_\Q^{\otimes 2k})^{\mathrm{Sp}}$ is defined
and its application to the tautological algebra of the moduli space of curves is given.
It turns out that this metric can be described as a direct consequence of a result of Hanlon and Wales 
\cite{hw}. In Section $2$, we first recall this metric and give a quick proof by quoting their result. 

Now we can consider $\mathfrak{h}_{g,1}(2k)^{\mathrm{Sp}}$ as a 
subspace of $(H_\Q^{\otimes (2k+2)})^{\mathrm{Sp}}$
so that it has the induced metric. To formulate our results,
we use the following terminology. A Young diagram $\lambda$ is denoted by $[\lambda_1\lambda_2\cdots\lambda_h]$
$(\lambda_1\geq \lambda_2\geq\cdots\geq\lambda_h)$
and the number of boxes in $\lambda$, namely $\lambda_1+\cdots+\lambda_h$, is denoted
by $|\lambda |$. We also denote the number of rows of $\lambda$, namely $h$ in the above
notation, by $h(\lambda)$. For a given Young diagram $\lambda$ as above, the symbol
$\lambda^\delta$ denotes another Young diagram 
$[\lambda_1\lambda_1\cdots \lambda_{h(\lambda)}\lambda_{h(\lambda)}]$
which has {\it multiple double floors} while the symbol $2 \lambda$ denotes 
$[2\lambda_1\cdots 2\lambda_{h(\lambda)}]$. Also $\mathfrak{S}_k$ denotes the 
symmetric group of order $k$.

\begin{thm}

With respect to the canonical metric on $\mathfrak{h}_{g,1}(2k)^\mathrm{Sp}$, there exists an orthogonal
direct sum decomposition
$$
\mathfrak{h}_{g,1}(2k)^\mathrm{Sp}\cong \bigoplus_{|\lambda |=k+1,\ h(\lambda)\leq g}
H_{\lambda}
$$
in terms of certain subspaces $H_\lambda$ which are indexed by
Young diagrams $\lambda=[\lambda_1\cdots \lambda_{h(\lambda)}]$ with $(k+1)$ boxes.
The dimension of $H_\lambda$ is given by the
following formula
$$
\mathrm{dim}\, H_\lambda =
\frac{1}{(2k+2)!} \sum_{\gamma\in \mathfrak{S}_{2k+2}}
\chi_{2k} (\gamma)\chi_{\la^\delta}(\gamma)
$$
where $\chi_{2k}$ is the character 
\begin{align*}
&\chi_{2k}(1^{2k+2})=(2k)!, \quad \chi_{2k}(1^1a^b)=(b-1)! \,a^{b-1} \mu(a)\ (\text{$\mu$: the M\"obius function})\\
&\chi_{2k}(a^b)=-(b-1)! \,a^{b-1} \mu(a),\quad
\chi_{2k}(\text{other conjugacy class})=0
\end{align*}
defined by Kontsevich and $\chi_{\lambda^\delta}$ denotes the character of the irreducible representation of
$\mathfrak{S}_{2k+2}$ corresponding to the Young diagram 
$\lambda^\delta=[\lambda_1\lambda_1\cdots \lambda_{h(\lambda)}\lambda_{h(\lambda)}]$ which has $(2k+2)$ boxes.
\label{thm:ortho}
\end{thm}

It follows that any $\mathrm{Sp}$-invariant element $\xi\in \mathfrak{h}_{g,1}(k)^{\mathrm{Sp}}$ can be expressed
as
$$
\xi= \sum_{|\lambda |=k+1,\ h(\lambda)\leq g}
\xi_\lambda \quad (\xi_\lambda\in H_{\lambda}).
$$
We call $\xi_\lambda$ the {\it $\lambda$-coordinate} of $\xi$.

\begin{thm}
The stable range with respect to the genus $g$ of the spaces $\mathfrak{h}_{g,1}(2k)^\mathrm{Sp}$ 
is given by
\begin{align*}
&\dim\, \mathfrak{h}_{k,1}(2k)^{\mathrm{Sp}}=\dim\, \mathfrak{h}_{k+1,1}(2k)^{\mathrm{Sp}}=\cdots\quad (\text{$k$ odd})\\
&\dim\, \mathfrak{h}_{k-1,1}(2k)^{\mathrm{Sp}}=\dim\, \mathfrak{h}_{k,1}(2k)^{\mathrm{Sp}}=\cdots\quad (\text{$k$ even}).
\end{align*}
Furthermore we have
$$
\dim\, \mathfrak{h}_{k,1}(2k)^{\mathrm{Sp}}=\dim\, \mathfrak{h}_{k-1,1}(2k)^{\mathrm{Sp}}+1
$$
for any odd $k\geq 3$.
\label{thm:stable}
\end{thm}

By combining the above theorem with our earlier computation given in \cite{mss3},
we have determined the orthogonal decomposition of $\mathfrak{h}_{g,1}(2k)^\mathrm{Sp}$
explicitly for all $2k\leq 20$ 
(see Table \ref{tab:h268}, Table \ref{tab:h1012} in Section $2$ and Tables \ref{tab:h14}-\ref{tab:h20} in Section $9$). 

It is a very important problem to determine whether a given element
in $\mathfrak{h}_{g,1}(k)^{\mathrm{Sp}}$ can be expressed as a linear
combination of brackets of elements of $\mathfrak{h}_{g,1}$ with
lower degrees and, in case it can be, we have a further problem of 
determining whether the difference between 
various expressions give rise to non-trivial elements in $H_2(\mathfrak{h}_{g,1})$
or not. In this regard, we obtain the following result.

For each Young diagram $\lambda$ with $k$ boxes, let $V_\lambda^k$ denote
the isotypical  component of $H_\Q^{\otimes k}$, considered as a  $\mathrm{GL}$-module,
corresponding to $\lambda$, namely it is
the sum of all the $\mathrm{GL}$-submodules of $H_\Q^{\otimes k}$ which are
isomorphic to the $\mathrm{GL}$-irreducible representation $\lambda_{\mathrm{GL}}$.
Then we can write
$$
H_\Q^{\otimes k}=\bigoplus_{|\lambda|=k } V_\lambda^k.
$$
We consider $\mathfrak{h}_{g,1}(k)$ to be a subspace of $H_\Q^{\otimes (k+2)}$ and
for each Young diagram $\lambda$ with $(k+2)$ boxes, we
set $\widetilde{H}_{\lambda}=V^{k+2}_\lambda\cap \mathfrak{h}_{g,1}(k)$.
Then we can write 
$$
\mathfrak{h}_{g,1}(k)=\bigoplus_{|\lambda|=k+2} \widetilde{H}_\lambda 
$$
where $\widetilde{H}_\lambda $ is
the totality of ${\mathrm{GL}}$-irreducible summands of $\mathfrak{h}_{g,1}(k)$
which are isomorphic to $\lambda_{\mathrm{GL}}$.  Hence we can write
$$
\widetilde{H}_\lambda \cong \lambda_{\mathrm{GL}}^{\oplus m_\lambda}
$$
where $m_\lambda$ denotes the multiplicity of $\lambda_{\mathrm{GL}}$ in $\mathfrak{h}_{g,1}(k)$.
We call $\widetilde{H}_\lambda $ the {\it $\lambda$-isotypical component} of $\mathfrak{h}_{g,1}(k)$.
In relation to Theorem \ref{thm:ortho}, we have the following.
If $\lambda$ is a Young diagram with $(k+1)$ boxes, then we have
$$
H_\lambda\cong \widetilde{H}_{\lambda^\delta}^{\mathrm{Sp}}
\quad \text{and}\quad 
\widetilde{H}_{\lambda^\delta}\cong H_\lambda\otimes \lambda^\delta_{\mathrm{GL}}.
$$

\begin{thm}
Let $\widetilde{H}_\lambda\subset \mathfrak{h}_{g,1}(k)$ be the $\lambda$-isotypical component 
of $\mathfrak{h}_{g,1}(k)$ and let $\widetilde{H}_\mu\subset \mathfrak{h}_{g,1}(\ell)$ be the $\mu$-isotypical component 
of $\mathfrak{h}_{g,1}(\ell)$ where $\lambda$ and $\mu$ denote Young diagrams with $(k+2)$ boxes
and $(\ell+2)$ boxes respectively. 

$\mathrm{(i)}\ $ The bracket $[\widetilde{H}_\lambda,\widetilde{H}_\mu]$
is included in a $\mathrm{GL}$-submodule $B(\lambda,\mu)$ of $\mathfrak{h}_{g,1}(k+\ell)$
which is defined by
$$
B(\lambda,\mu)=\bigoplus_{\text{$\nu$ satisfies $(C)$}} \widetilde{H}_{\nu}.
$$
Here the condition $(C)$ is given by
$$
\text{(C): $\left(\lambda_{\mathrm{GL}}\otimes \mu_{\mathrm{GL}}\right)$ and 
$\left(\wedge^2 H_\Q\otimes \nu_{\mathrm{GL}}\right)$
have a common $\mathrm{GL}$-irreducible summand.}
$$
In particular, the height of such a $\nu$ must satisfy the following inequality
$$
\mathrm{max}\, \{h(\lambda),h(\mu)\}-2 \leq h(\nu)\leq h(\lambda)+h(\mu).
$$

$\mathrm{(ii)}\ $
The $\mathrm{Sp}$-invariant part $[\widetilde{H}_\lambda,\widetilde{H}_\mu]^{\mathrm{Sp}}$
of $[\widetilde{H}_\lambda,\widetilde{H}_\mu]$
is included in a submodule $S(\lambda,\mu)$ of $\mathfrak{h}_{g,1}(k+\ell)^{\mathrm{Sp}}$
which is defined by
$$
S(\lambda,\mu)=\bigoplus_{\text{$\nu$ satisfies $(S)$}} H_{\nu}.
$$
Here the condition $(S)$ is given by
$$
\text{(S): $\left(\lambda_{\mathrm{GL}}\otimes \mu_{\mathrm{GL}}\right)$ and 
$\left(\wedge^2 H_\Q\otimes \nu^{\delta}_{\mathrm{GL}}\right)$
have a common $\mathrm{GL}$-irreducible summand.}
$$
In particular, the
$\nu$-coordinate of any element in
$[\widetilde{H}_\lambda,\widetilde{H}_\mu]^{\mathrm{Sp}}$ vanishes for all $\nu$ such that
$2 h(\nu)> h(\lambda)+h(\mu)$ or $2 h(\nu)< \mathrm{max}\, \{h(\lambda),h(\mu)\}-2$.
\label{thm:bracket}
\end{thm}

\begin{remark}
The above result shows that the intersection matrix for the pairings 
of the bracket operation 
$$
\sum_{i+j=k} \ \mathfrak{h}_{g,1}(i)\otimes \mathfrak{h}_{g,1}(j) \rightarrow \mathfrak{h}_{g,1}(k)
$$
is, so to speak, ``lower anti-triangular" with respect to the heights of Young diagrams 
which appear in the $\mathrm{GL}$-irreducible decompositions of $\mathfrak{h}_{g,1}(k)$'s.
Furthermore the non-vanishing area is a rather restricted one near the anti-diagonal.
\end{remark}

\begin{example}
Asada-Nakamura \cite{an} proved that there exists a unique copy 
$[2k+1\, 1^2]_{\mathrm{GL}}\subset \mathfrak{h}_{g,1}(2k+1)$ for any $k\geq 1$.
On the other hand, as an $\mathrm{Sp}$-module, we have an $\mathrm{Sp}$-irreducible decomposition
$$
[2k+1\, 1^2]_{\mathrm{GL}}=[2k+1\, 1^2]_{\mathrm{Sp}}\oplus [2k,1]_{\mathrm{Sp}}\oplus [2k+1]_{\mathrm{Sp}}
$$
for all $g\geq 3$.
The second wedge product of each of these three irreducible components gives rise to an
$\mathrm{Sp}$-invariant element contained in $\mathfrak{h}_{g,1}(4k+2)^{\mathrm{Sp}}$.
The above theorem implies that the $\lambda$-coordinate of this element
vanishes for all $\lambda$ with $h(\lambda)>3$. On the other hand, explicit computation 
for the case $k=1$ shows that the $[21^2]$-coordinate does not vanish so that 
the above theorem gives the best possible result.
\end{example}

Proofs of the above theorems are given in Section $2$. 

In Section $3$,
we compare the three Lie algebras $\mathfrak{h}_{g,1},\mathfrak{h}_{g,*},\mathfrak{h}_{g}$.
These Lie algebras are the {\it rational} forms of the corresponding Lie algebras
$\mathfrak{h}^\Z_{g,1},\mathfrak{h}^\Z_{g,*},\mathfrak{h}^\Z_{g}$ which are defined over $\Z$.
In the cases of the corresponding mapping class groups denoted by $\mathcal{M}_{g,1}, \mathcal{M}_{g,*}, \mathcal{M}_{g}$,
the relations among them are described by the following two well-known extensions
\begin{align*}
0\ \rightarrow \Z\ \rightarrow\ \mathcal{M}_{g,1}\ \rightarrow \mathcal{M}_{g,*}\ \rightarrow\ 1 ,\\
1\ \rightarrow \pi_1 \Sigma_g\ \rightarrow\ \mathcal{M}_{g,*}\ \rightarrow \mathcal{M}_{g}\ \rightarrow\ 1
\end{align*}
which hold for any $g\geq 2$. In the cases of the above Lie algebras over $\Z$, the relations are
described by the following extensions
\begin{align*}
0\ \rightarrow \mathfrak{j}^\Z_{g,1}\ \rightarrow\ &\mathfrak{h}^\Z_{g,1}\ \rightarrow \mathfrak{h}^\Z_{g,*}\ \rightarrow\ 0 ,\\
0\ \rightarrow \mathcal{L}^\Z_g\ \rightarrow\ &\mathfrak{h}^\Z_{g,*}\ \rightarrow \mathfrak{h}^\Z_{g}\ \rightarrow\ 0
\end{align*}
where $\mathfrak{j}^\Z_{g,1}$ is a certain ideal of $\mathfrak{h}^\Z_{g,1}$ and $\mathcal{L}^\Z_g$ denotes the 
Malcev Lie algebra, over $\Z$, of $\pi_1\Sigma_g$ (see \cite{morita99} and Section $2$ for more details). 

In Section $4$, we formulate a method
of describing the $\mathrm{Sp}$-decompositions of  the two Lie algebras
$\mathfrak{h}_{g,*},\mathfrak{h}_{g}$ which is based  on a
theorem of Labute \cite{labute}. 
\begin{thm}
Over the rationals, we have a direct sum decomposition
$$
\mathfrak{h}_{g,1}(k)\cong \mathfrak{j}_{g,1}(k)\oplus \mathcal{L}_g(k)\oplus \mathfrak{h}_g(k).
$$
Furthermore there exists an explicit method of determining the $\mathrm{Sp}$-irreducible decompositions of 
the $\mathrm{Sp}$-modules $\mathfrak{j}_{g,1}(k)$ and $\mathcal{L}_g(k)$.
\end{thm}
See Theorem \ref{thm:jL} for the precise decompositions mentioned here.
By applying this method and extending our results of \cite{mss3},
we obtain explicit $\mathrm{Sp}$-decompositions of $\mathfrak{j}_{g,1}(k)$, $\mathcal{L}_g(k)$ 
and $\mathfrak{h}_g(k)$ for all $k\leq 20$. See Section $4$ for details.

If we apply various contractions to 
any $\mathrm{GL}$-irreducible summand of $\mathrm{GL}$-irreducible
decomposition of $\mathfrak{h}_{g,1}(k)$, then we obtain
various $\mathrm{Sp}$-irreducible components. Of course non-isomorphic
$\mathrm{GL}$-irreducible summands may produce isomorphic 
$\mathrm{Sp}$-irreducible components. Keeping this fact in mind, to analyze the structure of the
Lie algebra $\mathfrak{h}_{g,1}$, we propose to take 
the process of contractions into consideration. We 
formulate this idea in Section $5$ under the names of {\it descendants} and {\it ancestors}.
See Definition \ref{def:da} and Theorem \ref{thm:imtau}.

In Section $6$  we give a general method of constructing elements of $\mathfrak{h}_{g,1}^{\mathrm{Sp}}$
and by using it, we reveal a considerable difference in property 
between the {\it $\mathrm{Sp}$-invariant parts} 
of the two Lie subalgebras of $\mathfrak{h}_{g,1}$,
one is the ideal
$\mathfrak{j}_{g,1}$ and the other is the Johnson image $\mathrm{Im}\, \tau_{g,1}$.
See Theorem \ref{thm:imtauw}.

By making use of the above results, together with the Enomoto-Satoh map given in \cite{es},
we have extended known results considerably to obtain a complete description of the structure of 
the Lie algebra $\mathfrak{h}_{g,1}$ 
up to degree $6$.  It is summarized in the following theorem and we give more detailed
structure theorem for $\mathfrak{h}_{g,1}(6)^{\mathrm{Sp}}$ in Section 7 (see Theorem \ref{thm:h6sp}).


\begin{thm}
The structure of the Lie algebra $\mathfrak{h}_{g,1}$ up to degree $6$ is as in 
Table \ref{tab:6} where the symbol with double parentheses, e.g. $[[3]]$, means that it
remains in the abelianization $H_1(\mathfrak{h}^+_{g,1})$.
\label{thm:t6}
\end{thm}
\begin{table}[h]
\caption{$\text{List of $\mathfrak{h}_{g,1}(k), \mathfrak{j}_{g,1}(k),
\mathcal{L}_{g}(k),\mathrm{Im}\,\tau_g(k), \mathrm{Coker}\,\tau_g(k)$}$}
\begin{center}
\begin{tabular}{|r|r|r|r|r|r|r|}
\noalign{\hrule height0.8pt}
\hfil $k$ & $\mathfrak{h}_{g,1}(k)$ & $\mathfrak{j}_{g,1}(k)$ & $\mathcal{L}_{g}(k)$  & 
$\mathrm{Im}\,\tau_g(k)$ & $\mathrm{Coker}\,\tau_g(k)$ \\
\hline
$1$ & $[1^3][1]$& $$ &  $[1]$   & $[1^3]$  & $$\\
\hline
$2$ & $[2^2][1^2][0]$ & $[0]$& $[1^2]$ & $[2^2]$ & $$\\
\hline
$3$ & $[31^2][21][3]$ & $$ &  $[21]$ & $[31^2]$ & $[[3]]$\\
\hline
$4$ & $[42][31^3][2^3]$\ & $[2]$  & $[31][21^2][2]$ 
& $[42][31^3][2^3]$   & $[21^2]$ \\
{} & $2[31]2[21^2]3[2]$ & {}  & {} & $[31][2]$ & {} \\
\hline
$5$ & $[51^2][421][3^21]$ & $[21][1^3][1]$ & $[41][32][31^2]$  & $[51^2][421][3^21]$ & $[[5]][32]$\\
{} & $[321^2][2^21^3][5]$ & $$  & $[2^21][21^3][3]$ & $[321^2][2^21^3]$ & $[2^21][1^5]$\\
{} & $2[41]3[32]2[31^2]$ & {}  & $2[21][1^3][1]$ & $[41][32][31^2]$ & $[21][1^3]$ \\
{} & $3[2^21]2[21^3][1^5]$ & {}  & $$ & $[2^21][21^3]$ & $[1]$ \\
{} & $[3]5[21]4[1^3]3[1]$ & {}  & $$ & $[21][1^3]$ & {} \\
\hline
$6$ & $[62][521][51^3]$  & $[4][31]2[2^2]$ & $[51][42]2[41^2]$ & $[62][521][51^3]$ & $2[41^2][3^2][321]$\\   
$$ & $[431][4^2]2[42^2]$ & $[21^2][1^4]$ & $[3^2]3[321][31^3]$ & $[431][4^2]2[42^2]$ & $[31^3][2^21^2]$\\
$$ & $[421^2][41^4]2[3^21^2]$ & $[2]3[1^2]2[0]$ & $2[2^21^2][21^4]$ & $[421^2][41^4]2[3^21^2]$ & $[4][[31]][31][2^2]$\\
$$ & $[32^21][321^3]$ & $$ & $4[31]4[2^2]2[4]$ & $[32^21][321^3]$ & $2[21^2][1^4]2[1^2][0]$\\
$$ & $[2^4]3[51]3[42]$ & $$ & $4[21^2]2[1^4]2[2]$ & $[2^4]2[51]2[42]$ & $$\\
$$ & $6[41^2]4[3^2]8[321]$ & $$ & $4[1^2][0]$ & $2[41^2]2[3^2]4[321]$ & $$\\
$$ & $4[31^3][2^3][2^21^4]$ & $$ & $$ & $2[31^3][2^3][2^21^4]$ & $$\\
$$ & $6[2^21^2]2[21^4][1^6]$ & $$ & $$ & $3[2^21^2][21^4][1^6]$ & $$\\
$$ & $6[4]9[31]12[2^2]$ & $$ & $$ & $2[4]2[31]5[2^2]$ & $$\\
$$ & $9[21^2]6[1^4]4[2]$ & $$ & $$ & $2[21^2]2[1^4][2]$ & $$\\
$$ & $11[1^2]5[0]$ & $$ & $$ & $2[1^2][0]$ & $$\\   
\noalign{\hrule height0.8pt}
\end{tabular}
\end{center}
\label{tab:6}
\end{table}
In Table \ref{tab:6}, the symbol $[1^3]$ denotes the $\mathrm{Sp}$-irreducible representation corresponding to the
Young diagram $[1^3]=[111]$. Also the symbol $\tau_g(k)$ denotes the Johnson homomorphism
$\tau_g(k): \mathcal{M}_g(k)\rightarrow \mathfrak{h}_g(k)$ for {\it closed} surface
where $\{\mathcal{M}_g(k)\}_k$ denotes the Johnson filtration for the mapping class group $\mathcal{M}_g$.
Details are given in Section $7$.

In  Section $8$, we study the case of genus $1$. This is motivated by the theory of 
universal mixed elliptic motives due to
Hain and Matsumoto (cf. \cite{hain14}). They show, among other things,  that 
certain Galois obstructions appear in the $\mathrm{Sp}$-invariant part $\mathfrak{h}_{1,1}^{\mathrm{Sp}}$.
We begin to study these elements from our point of view
(see Theorem \ref{thm:z5}). Here the theory of 
Satoh \cite{satoh} and Enomoto and Satoh \cite{es} play an important role.

Finally we mention two reasons why we put an emphasis on
the $\mathrm{Sp}$-invariant part $\mathfrak{h}^{\mathrm{Sp}}_{g,1}$ in our study of the whole Lie algebra $\mathfrak{h}_{g,1}$.
One is that the Galois obstructions, predicted by Oda
and proved by Nakamura \cite{nakamura} and Matsumoto \cite{matsumoto}
independently, appear in this $\mathrm{Sp}$-invariant Lie subalgebra. 
See a recent survey article \cite{matsumotopcmi} of Matsumoto.
The Galois obstructions
also appear in the genus $1$ symplectic derivation algebra 
as already mentioned above.
The precise description of these obstructions is still a mystery and it should be a very important problem
both in number theory and topology.
See also Willwacher \cite{willwacher} for a related work.

The other concerns another important problem of deciding whether the composition
$$
\mathfrak{h}^{\mathrm{Sp}}_{g,1}\subset \mathfrak{h}_{g,1}\rightarrow \lim_{g\to\infty}\ H_1(\mathfrak{h}_{g,1})
$$
is trivial. This is equivalent to the vanishing of the top homology group of
$\mathrm{Out}\, F_n$ with respect to its virtual cohomological dimension
determined by Culler and Vogtmann \cite{cuv}
(see Conjecture 1.3. and Remark 1.5. in \cite{mss3}, and as for a recent result on $H_1(\mathfrak{h}^+_{g,1})$,
see \cite{ckv}).
This is also related to one more mystery in low dimensional topology
because we can show that there exists a surjective homomorphism
$H_1(\mathcal{H}_{g,1};\Q)\rightarrow H_1(\mathfrak{h}_{g,1})$
where $\mathcal{H}_{g,1}$ denotes the group of homology cobordism classes of homology
cylinders introduced by Garoufalidis and Levine \cite{gl}. It should be a very important
problem to determine whether $H_1(\mathcal{H}_{g,1};\Q)=0$ or not.
Recall here that Cha, Friedl and Kim \cite{cfk} proved that $H_1(\mathcal{H}_{g,1};\Z)$
contains $(\Z/2)^{\infty}$ as a direct summand.

{\it Acknowledgement} The authors would like to thank Naoya Enomoto and Takao Satoh for 
enlightening discussion
about the cokernel of the Johnson homomorphisms. 
Thanks are also due to Richard Hain,
Makoto Matsumoto and Hiroaki Nakamura for helpful
information about the arithmetic mapping class groups.
The first named author would like to thank Dan Petersen for 
informing him about a paper \cite{hw} of Hanlon and Wales.
Finally we would like to thank the referee for helpful comments.

The authors were partially supported by KAKENHI (No.~24740040 and 
No.~24740035), 
Japan Society for the Promotion of Science, 
Japan.

\section{Canonical metric on $(H_\Q^{\otimes 2k})^{\mathrm{Sp}}$ and orthogonal decomposition of 
$\mathfrak{h}_{g,1}^{\mathrm{Sp}}$}\label{sec:ortho}

In this section, we first recall the canonical metric on $(H_\Q^{\otimes 2k})^{\mathrm{Sp}}$  introduced in \cite{morita13},
together with a few related facts, and then
we give proofs of Theorems \ref{thm:ortho},\ref{thm:stable} and \ref{thm:bracket}.

As already mentioned in the introduction, the above canonical metric can be described 
as a direct consequence of 
a result of Hanlon and Wales \cite{hw}. 
In order to make this paper self-contained as far as possible,
we give a quick proof of the main property of this metric 
by quoting their result in addition to our own arguments,
while we refer to \cite{morita13} for more geometric proof.

We begin by recalling a few facts from \cite{morita99}.
Let
$
\mathcal{D}^{\ell}(2k)
$
denote the set consisting of all the linear chord diagrams with $2k$ vertices.
Here a linear chord diagram with $2k$ vertices is a partition of the set
$\{1,2,\ldots,2k\}$ into $k$-tuple 
$$
C=\{\{i_1,j_1\},\ldots,\{i_k,j_k\}\}\quad 
$$
of pairs $\{i_l,j_l\}\ (l=1,\ldots,k)$ where we assume 
$$
i_1<\cdots<i_k, \ i_l<j_l \quad (l=1,\ldots,k).
$$
We may also adopt a simpler notation $C=(i_1 \, j_1)\cdots (i_k \, j_k)$ than the above.
We consider $C$ to be a graph with $2k$ vertices and $k$ edges (chords)
each of which connects $i_\ell$ with $j_\ell$ for $\ell=1,\cdots,k$.
This set  $\mathcal{D}^{\ell}(2k)$ has $(2k-1)!!$ elements and let $\Q\mathcal{D}^{\ell}(2k)$ be the vector
space over $\Q$ spanned by $\mathcal{D}^{\ell}(2k)$.
For each linear chord diagram $C\in \mathcal{D}^\ell(2k)$, define
$$
a_C\in (H_\Q^{\otimes 2k})^{\mathrm{Sp}}
$$
by permuting the elements $(\omega_0)^{\otimes 2k}$ in such a way that the $s$-th part
$(\omega_0)_s$ of this tensor product goes to $(H_\Q)_{i_s}\otimes (H_\Q)_{j_s}$, where $(H_\Q)_{i}$
denotes the $i$-th component of $H_\Q^{\otimes 2k}$, and then multiplied by the factor
$$
\mathrm{sgn}\, C=\mathrm{sgn}
\begin{pmatrix}
1 & 2 & \cdots & 2k-1 & 2k\\
i_1 & j_1 & \cdots & i_k & j_k
\end{pmatrix}
.
$$
In a dual setting, we can also define
$$
\alpha_C\in \mathrm{Hom}(H_\Q^{\otimes 2k},\Q)^{\mathrm{Sp}}
$$
by 
$$
\alpha_C(u_1\otimes\cdots\otimes u_{2k})=\mathrm{sgn}\, C
\prod_{s=1}^k u_{i_s}\cdot u_{j_s}\ (u_i\in H_\Q)
$$
where $u_{i_s}\cdot u_{j_s}$ denotes the intersection number of $u_{i_s}$ and $u_{j_s}$ as before.

Now define a linear mapping
$$
\Phi: \Q\mathcal{D}^\ell (2k) \rightarrow (H_\Q^{\otimes 2k})^{\mathrm{Sp}}
$$
by setting $\Phi(C)=a_C$. Observe that the symmetric  group $\mathfrak{S}_{2k}$
acts on both of $\Q\mathcal{D}^\ell (2k)$ and $(H_\Q^{\otimes 2k})^{\mathrm{Sp}}$
naturally.

\begin{prop}
The correspondence 
$$
\Phi: \Q\mathcal{D}^\ell (2k) \rightarrow (H_\Q^{\otimes 2k})^{\mathrm{Sp}}
$$
is surjective for any $g$ and bijective for any $g\geq k$. Furthermore this correspondence
is ``anti" $\mathfrak{S}_{2k}$-equivariant in the sense that
$$
\Phi(\gamma (C))=( \mathrm{sgn}\,\gamma) \  \gamma(\Phi(C))
$$
for any $C\in \mathcal{D}^\ell (2k)$ and $\gamma\in\mathfrak{S}_{2k}$.
\label{prop:LCD}
\end{prop}
\begin{proof}
The surjectivity follows from a classical result of Weyl on symplectic invariants (see \cite{morita99}).
The latter part can be shown by comparing the actions of permutations
on the symplectic invariant tensors and the linear chord diagrams.
\end{proof}


Next consider the following two symmetric bilinear pairings
\begin{align*}
\mu^{\otimes 2k}: H_\Q^{\otimes{2k}}&\otimes H_\Q^{\otimes{2k}}\rightarrow \Q\\
\langle\ \, ,\ \rangle: \Q\mathcal{D}^{\ell}(2k)&\times \Q\mathcal{D}^{\ell}(2k)\rightarrow\Q[g]
\end{align*}
which are defined by
\begin{align*}
\mu^{\otimes 2k}((u_1\otimes\cdots\otimes u_{2k})&\otimes (v_1\otimes\cdots\otimes v_{2k}))
=\prod_{i=1}^{2k}\, u_i\cdot v_i \quad (u_i, v_i\in H_\Q)\\
\langle C, C' \rangle&= (-1)^{k-r} (2g)^r\quad (C, C'\in \mathcal{D}^\ell(2k))
\end{align*}
where $r$ denotes the number of connected components of the graph $C\cup C'$
(here we consider interiors of $2k$ edges to be mutually {\it disjoint}). 
It is not difficult to check that the above two bilinear pairings correspond to each other
under the correspondence $\Phi$, namely we have the following result (see \cite{morita99}\cite{morita03}\cite{morita13}).
\begin{prop}
The following diagram is commutative
$$
\begin{CD}
\Q\mathcal{D}^\ell (2k) \otimes \Q\mathcal{D}^\ell (2k)  @>{\langle\ ,\ \rangle}>> \Q[g]\\
@V{\Phi\otimes\Phi}VV @VV{\text{evaluation}}V\\
(H_\Q^{\otimes 2k})^{\mathrm{Sp}}\otimes (H_\Q^{\otimes 2k})^{\mathrm{Sp}} @>{\mu^{\otimes 2k}}>> \Q.
\end{CD}
$$
\label{prop:lcdmu}
\end{prop}

To describe the canonical metric, we prepare one more terminology.
\begin{definition}
For each Young diagram $\lambda$,
let $\mu_\lambda$ be the number defined by the following formula
$$
\mu_\lambda=\prod_{\text{b: box of $\lambda$}} (2g-2s_b+t_b)
$$
where $s_b$ denotes the number of columns of $\lambda$ 
which are on the left of the column containing $b$ and  
$t_b$ denotes the number of rows which are above the 
row containing $b$.
When we specify the genus $g$, we write
$\mu_{\lambda'}(g)$ for $\mu_{\lambda'}$ which is a polynomial in $g$
of degree $|\lambda|$.
\label{def:ev}
\end{definition}

\begin{example}
$$
\mu_{[4]}=2g(2g-2)(2g-4)(2g-6),\quad
\mu_{[1^4]}=2g(2g+1)(2g+2)(2g+3).
$$
\end{example}

\begin{thm}
The intersection matrix of the symmetric bilinear pairing $\langle\ \, ,\ \rangle$
on $\Q\mathcal{D}^\ell (2k)$, with respect to the basis $\mathcal{D}^\ell (2k)$,
is positive semi-definite for any $g$,
so that it defines a metric on
$\Q\mathcal{D}^{\ell}(2k)$ depending on $g$.
Furthermore, there exists an orthogonal
direct sum decomposition
$$
\Q\mathcal{D}^{\ell}(2k)\cong \bigoplus_{|\lambda |=k}
E_{\lambda}
$$
in terms of eigenspaces $E_\lambda$ with eigenvalues
$\mu_\lambda$. This decomposition is
defined independently of the genus $g$.
With respect to the natural action of $\mathfrak{S}_{2k}$
on $\Q\mathcal{D}^{\ell}(2k)$, $E_\lambda$
is an irreducible $\mathfrak{S}_{2k}$-submodule isomorphic to
$
(2\lambda)_{\mathfrak{S}_{2k}}
$
so that the above orthogonal
direct sum decomposition gives also the irreducible decomposition of the
$\mathfrak{S}_{2k}$-module $\Q\mathcal{D}^{\ell}(2k)$.

\label{thm:lcdortho}
\end{thm}
\begin{proof}
The space $\Q\mathcal{D}^{\ell}(2k)$ as well as the bilinear pairing on it,
but with a different sign convention, appeared in the paper
\cite{hw} by Hanlon and Wales in a broader context of Brauer's centralizer algebras.
More precisely, they call the linear chord diagram in this paper {\it $1$-factor} and
they work over the reals $\R\mathcal{D}^{\ell}(2k)$. Also the bilinear pairings are indexed by
a real parameter $x$.  They determined the eigenspaces and eigenvalues of the intersection
matrix and also the irreducible decomposition of $\R\mathcal{D}^{\ell}(2k)$ as an $\mathfrak{S}_{2k}$-module.
In order to adapt their result to our context, we have to put the real parameter to be $x=-2g$ and then
take multiplication by the factor $(-1)^k$. Then their formula for the eigenvalues becomes the one
given in the above Definition \ref{def:ev}.

Thus, to prove the theorem, it remains to show the semi-positivity and orthogonality.
The former property follows from the semi-positivity of the adapted eigenvalues as in
Definition \ref{def:ev}. The latter property of orthogonality follows from the fact that the correspondence
$\lambda\mapsto\mu_\lambda$ from the set of Young diagrams to the set of eigenvalues,
considered as polynomials in $g$, is {\it injective}. This is because all the eigenvalues are then
different each other. This last claim is proved in \cite{morita13}.
\end{proof}

We now state the canonical metric on $(H_\Q^{\otimes 2k})^{\mathrm{Sp}}$.

\begin{thm}
For any $g$, the symmetric bilinear form $\mu^{\otimes 2k}$ on $(H_\Q^{\otimes 2k})^{\mathrm{Sp}}$
is positive definite so that it defines a metric on this space. 
Furthermore, there exists an orthogonal
direct sum decomposition
$$
(H_\Q^{\otimes 2k})^{\mathrm{Sp}}\cong \bigoplus_{|\lambda |=k,\ h(\lambda)\leq g}
U_{\lambda}
$$
in terms of eigenspaces $U_\lambda$ with eigenvalues
$\mu_{\lambda'}$.
With respect to the natural action of $\mathfrak{S}_{2k}$
on $(H_\Q^{\otimes 2k})^{\mathrm{Sp}}$, $U_\lambda$
is an irreducible $\mathfrak{S}_{2k}$-submodule isomorphic to
$
(\lambda^\delta)_{\mathfrak{S}_{2k}}
$
so that the above orthogonal
direct sum decomposition gives also the irreducible decomposition of the
$\mathfrak{S}_{2k}$-module $(H_\Q^{\otimes 2k})^{\mathrm{Sp}}$.
\label{thm:can}
\end{thm}
\begin{proof}
Proposition \ref{prop:lcdmu} shows that the homomorphism $\Phi$ converts all the
relevant 
structures of $\Q\mathcal{D}^{\ell}(2k)$ into those of $(H_\Q^{\otimes 2k})^{\mathrm{Sp}}$
{\it anti} $\mathfrak{S}_{2k}$-equivariantly. Also it is easy to see that the kernel of 
$\Phi$, for each specific value of $g$, is precisely the direct sum of
eigenspaces whose eigenvalues are zero. Hence the claim follows 
from Theorem \ref{thm:lcdortho} by applying Proposition \ref{prop:lcdmu}.
\end{proof}

Now we are ready to prove Theorems \ref{thm:ortho},\ref{thm:stable} and \ref{thm:bracket}.

\begin{proof}[Proof of Theorem $\ref{thm:ortho}$]
As already mentioned, the space $\mathfrak{h}_{g,1}(2k)^{\mathrm{Sp}}$ can be considered to be a subspace of
$\left(H_\Q^{\otimes (2k+2)}\right)^{\mathrm{Sp}}$.
By Theorem \ref{thm:can},
we have an orthogonal
direct sum decomposition
$$
\left(H_\Q^{\otimes (2k+2)}\right)^{\mathrm{Sp}}\cong \bigoplus_{|\lambda |=k+1,\ h(\lambda)\leq g}
U_{\lambda}
$$
and as an $\mathfrak{S}_{2k+2}$-submodule of 
$\left(H_\Q^{\otimes (2k+2)}\right)^{\mathrm{Sp}}$, $U_\lambda$
is an irreducible $\mathfrak{S}_{(2k+2)}$-module corresponding to the Young diagram
$
\lambda^\delta.
$
More precisely, 
$\mathrm{GL}$-irreducible representation $\lambda^\delta_{\mathrm{GL}}$ appears 
in the $\mathrm{GL}$-irreducible decomposition of $H_\Q^{\otimes (2k+2)}$ if and only
if $h(\lambda)\leq g$ and in this case its 
multiplicity is equal to the dimension of $\lambda^\delta_{\mathfrak{S}_{2k+2}}$.
On the other hand, each copy of $\lambda^\delta_{\mathrm{GL}}$ has a unique
$\mathrm{Sp}$-invariant element (up to scalars) and the totality of these 
$\mathrm{Sp}$-invariant elements makes a basis of the subspace $U_\lambda$.

Now it was shown in \cite{morita99} (see also Proposition \ref{prop:chah} below) that 
the space $\mathfrak{h}_{g,1}(2k)$ is the image of certain projecting operator acting on
$H_\Q^{\otimes (2k+2)}$ which is an element of $\Z[\mathfrak{S}_{2k+2}]$.
More precisely, we have
$$
S_{2k+2}\circ (1\otimes p_{2k+1}) (H_\Q^{\otimes (2k+2)})=\mathfrak{h}_{g,1}(2k).
$$
By restricting to the $\mathrm{Sp}$-invariant subspaces, we have
$$
S_{2k+2}\circ (1\otimes p_{2k+1}) \left(H_\Q^{\otimes (2k+2)}\right)^{\mathrm{Sp}}=\mathfrak{h}_{g,1}(2k)^{\mathrm{Sp}}.
$$
Since this operator is described in terms of the action of the symmetric group, 
it sends each subspace $U_\lambda$ to itself. 
Hence if we define
$$
H_\lambda=S_{2k+2}\circ (1\otimes p_{2k+1})(U_\lambda)
$$
we obtain the required orthogonal decomposition. Finally, the formula for the dimension of $H_\lambda$ follows
by applying Corollary 3.2 of \cite{mss3} because $\dim\, H_\lambda$ is equal to
the multiplicity of $\lambda^\delta_{\mathrm{GL}}$ in $\mathfrak{h}_{g,1}(2k)$. This completes the proof.
\end{proof}

As was mentioned in our paper \cite{mss3}, we have determined the $\mathrm{Sp}$-irreducible decompositions of
$\mathfrak{h}_{g,1}(k)$ for all $k\leq 20$. By combining this with Theorem \ref{thm:ortho}, we have determined
the orthogonal decompositions of $\mathfrak{h}_{g,1}(2k)^{\mathrm{Sp}}$ for all $2k\leq 20$.
The results are given in Tables \ref{tab:h268} and \ref{tab:h1012} below and further Tables \ref{tab:h14} - \ref{tab:h20} are given
in Section \ref{sec:tables}. Here the symbol $2[31]^\delta$ in Table \ref{tab:h268}, for example, means that  there are two copies of
the representation $[31]^\delta_{\mathrm{GL}}=[3^21^2]_{\mathrm{GL}}$ appear in $\mathfrak{h}_{g,1}(6)$
and the unique $\mathrm{Sp}$-invariant element in each copy contributes to
$\mathfrak{h}_{g,1}(6)^{\mathrm{Sp}}$.


\vspace{2cm}

\begin{table}[h]
\caption{$\text{Orthogonal decompositions of $\mathfrak{h}_{g,1}(2)^{\mathrm{Sp}}, \mathfrak{h}_{g,1}(6)^{\mathrm{Sp}}, \mathfrak{h}_{g,1}(8)^{\mathrm{Sp}}$}$}
\begin{center}
\begin{tabular}{|r|r|r|r|r|r|r|}
\noalign{\hrule height0.8pt}
\hfil {} & {} & $\mathfrak{h}_{g,1}(2)^{\mathrm{Sp}}$& {} & $\mathfrak{h}_{g,1}(6)^{\mathrm{Sp}}$ & 
{} & $\mathfrak{h}_{g,1}(8)^{\mathrm{Sp}}$   \\
\hline
{} & $\mathrm{dim}$& $\text{eigenspace}$ & $\mathrm{dim}$& $\text{eigenspaces}$ & 
$\mathrm{dim}$ & $\text{eigenspaces}$    \\
\hline
$g=1$ & $1$ & $[2]^\delta$ & $1$ & $[4]^\delta$ & $0$ & {} \\
\hline
$g=2$ & {} & {} & $4$ & $2[31]^\delta [2^2]^\delta$ & $2$ & $[41]^\delta [32]^\delta$\\
\hline
$g\geq 3$ & {} & {} & $5$ & $[21^2]^\delta$ &
$3$ & $[31^2]^\delta$  \\                                                
\noalign{\hrule height0.8pt}
\end{tabular}
\end{center}
\label{tab:h268}
\end{table}


\begin{table}[h]
\caption{$\text{Orthogonal decompositions of $\mathfrak{h}_{g,1}(10)^{\mathrm{Sp}}, \mathfrak{h}_{g,1}(12)^{\mathrm{Sp}}$}$}
\begin{center}
\begin{tabular}{|r|r|r|r|r|}
\noalign{\hrule height0.8pt}
\hfil {} & {} & $\mathfrak{h}_{g,1}(10)^{\mathrm{Sp}}\hspace{7mm}$ & 
{} & $\mathfrak{h}_{g,1}(12)^{\mathrm{Sp}}\hspace{2cm}$   \\
\hline
{} & $\mathrm{dim}$& $\text{eigenspaces}\hspace{7mm}$ & 
$\mathrm{dim}$ & $\text{eigenspaces}\hspace{2cm}$    \\
\hline
$g=1$ & $3$ & $3[6]^\delta$ & $0$ & {} \\
\hline
$g=2$ & $51$ & $15[51]^\delta 26[42]^\delta 7[3^2]^\delta$ & $190$ & $31[61]^\delta 103[52]^\delta 56[43]^\delta$\\
\hline
$g=3$ & $97$ & $19[41^2]^\delta 24[321]^\delta 3[2^3]^\delta$ &
$97$ & $68[51^2]^\delta 216[421]^\delta 42[3^21]^\delta 28[32^2]^\delta$  \\
\hline
$g=4$ & $107$ & $7[31^3]^\delta 3[2^21^2]^\delta$  & $643$ & $36[41^3]^\delta 60[321^2]^\delta 3[2^31]^\delta$\\
\hline
$g\geq 5$ & $108$ & $[21^4]^\delta$ & $650$ & $5[31^4]^\delta 2[2^21^3]^\delta$ \\                                                  
\noalign{\hrule height0.8pt}
\end{tabular}
\end{center}
\label{tab:h1012}
\end{table}

\begin{proof}[Proof of Theorem $\ref{thm:stable}$]
It is easy to see that the stable range of $(H_\Q^{\otimes 2k})^{\mathrm{Sp}}$ is $g\geq k$.
On the other hand, $\mathfrak{h}_{g,1}(2k)^{\mathrm{Sp}}$
is a submodule of $(H_\Q^{\otimes (2k+2)})^{\mathrm{Sp}}$ so that the stable range of 
$\mathfrak{h}_{g,1}(2k)^{\mathrm{Sp}}$ is smaller than or equal to $g\geq k+1$.
Now, as was mentioned in \cite{mss3}, Proposition 4.1, for any Young diagram 
$\lambda$
$$
\dim\, (\lambda_{\mathrm{GL}})^{\mathrm{Sp}}=
\begin{cases}
&1\quad (\text{$\lambda$: multiple double floors, namely $\lambda=\mu^\delta$ for some $\mu$})\\
&0\quad (\text{otherwise}).
\end{cases}
$$
Let $\lambda=[\lambda_1,\cdots,\lambda_h]$ be a Young diagram with $(2k+2)$ boxes 
and $h$ rows and assume that it is with multiple double floors. Then it is easy to see the following:
\begin{align*}
& h=2k+2\ \Rightarrow\ \lambda=[1^{2k+2}],\\
& h=2k\ \Rightarrow\ \lambda=[2^2 1^{2k-2}],\\
& h=2k-2\ \Rightarrow\  \lambda=[3^2 1^{2k-4}]\ \text{or}\ [2^4 1^{2k-6}].
\end{align*}
Hence, to prove the result it is enough to show the following four facts:
\begin{align*}
\mathrm{(i)}\ &\text{$[1^{2k+2}]_{\mathrm{GL}}$ does not appear in the $\mathrm{GL}$-irreducible decomposition of $\mathfrak{h}_{g,1}(2k)$},\\
\mathrm{(ii)}\ &\text{$[2^2 1^{2k-2}]_{\mathrm{GL}}$ appears in the $\mathrm{GL}$-irreducible decomposition of $\mathfrak{h}_{g,1}(2k)$}\\
& \hspace{2mm} \text{with multiplicity $1$ for any odd $k\geq 1$},\\
\mathrm{(iii)}\ &\text{$[2^2 1^{2k-2}]_{\mathrm{GL}}$ does not appear in the $\mathrm{GL}$-irreducible decomposition of $\mathfrak{h}_{g,1}(2k)$}\\
& \hspace{2mm} \text{for any even $k$},\\
\mathrm{(iv)}\ &\text{$[3^2 1^{2k-4}]_{\mathrm{GL}}$ appears in the $\mathrm{GL}$-irreducible decomposition of $\mathfrak{h}_{g,1}(2k)$}\\
& \hspace{2mm} \text{for any even $k\geq 4$ with non-zero multiplicity}.
\end{align*}
$\mathrm{(i)}$ follows easily. Indeed $[1^{2k+2}]_{\mathrm{GL}}$ is nothing other than the alternating product
$\wedge^{2k+2} H_\Q$ which is {\it not} invariant under the $\Z/(2k+2)$-cyclic permutation of 
$H_\Q^{\otimes (2k+2)}$ while any summand of $\mathfrak{h}_{g,1}(2k)$ must be so.
$\mathrm{(ii)}$ and $\mathrm{(iii)}$ were proved by Enomoto and Satoh in \cite{es}. Indeed,
they proved that $[2^2 1^{k-2}]_{\mathrm{GL}}$ appears in $\mathfrak{h}_{g,1}(k)$ with multiplicity
$1$ for any $k$ such that $k\equiv 1\ \text{or}\ 2\ \mathrm{mod}\, 4$ and does not appear otherwise.
Finally $\mathrm{(iv)}$ follows from an explicit computation using the method of \cite{mss3} but here
we omit it.
\end{proof}


To prove Theorem \ref{thm:bracket}, we prepare the following.

\begin{lem}
Let
$$
K_{ij}: H_\Q^{\otimes (k+2)}\rightarrow H_\Q^{\otimes k}\quad (1\leq i<j\leq k+2)
$$
be a linear mapping defined by
$$
K_{ij}(u_1\otimes u_2\otimes\cdots\otimes u_{k+2})=(u_i\cdot u_j)\, u_1\otimes\cdots\otimes
\hat{u}_i \otimes\cdots\otimes\hat{u}_j \otimes\cdots\otimes u_{k+2}
\quad (u_i\in H_\Q)
$$
where $u_i \cdot u_j$ denotes the intersection number of $u_i$ and $u_j$, 
and the symbol $\hat{u}_i$ means that we delete $u_i$.
Let $V\subset H_\Q^{\otimes (k+2)}$ be an irreducible $\mathrm{GL}$-submodule which
is isomorphic to $\lambda_{\mathrm{GL}}$ where $\lambda$ is a Young diagram with
$(k+2)$ boxes. Consider the following condition $(C)$ posed on Young diagrams $\nu$
with $|\nu|=k$
$$
\text{(C): $\left(\wedge^2 H_\Q\otimes \nu_{\mathrm{GL}}\right)$
has a direct summand isomorphic to $\lambda_{\mathrm{GL}}$}.
$$
Then we have 
$$
K_{ij}(V)\ \subset \bigoplus_{\text{$\nu$ satisfies (C)}} V^k_{\nu}.
$$
\label{lem:K}
\end{lem}

\begin{proof}
Since the natural action of the symmetric group $\mathfrak{S}_{k+2}$ on $H_\Q^{\otimes (k+2)}$
preserves the structure of a $\mathrm{GL}$-module on it, it is enough to prove the case  $i=1, j=2$.
Now consider the following direct sum decomposition
$$
H_\Q^{\otimes (k+2)}=\left(S^2H_\Q\otimes H_\Q^{\otimes k}\right)\oplus
\left(\wedge^2H_\Q\otimes H_\Q^{\otimes k}\right).
$$
Since the mapping $K_{12}$ is trivial on the former summand, if we denote by 
$q: H_\Q^{\otimes (k+2)}\rightarrow \wedge^2H_\Q\otimes H_\Q^{\otimes k}$
the natural projection onto the latter summand, then we have
$$
K_{12}=K_{12}\circ q.
$$
It follows that $K_{12}(V)=K_{12}(q(V))$. Since $q(V)$ is a $\mathrm{GL}$-submodule of 
$\wedge^2H_\Q\otimes H_\Q^{\otimes k}$ isomorphic to $\lambda_{\mathrm{GL}}$, 
in view of the definition of the condition (C), we 
can conclude that
$$
q(V)\ \subset \wedge^2H_\Q\otimes\left(\bigoplus_{\text{$\nu$ satisfies (C)}} V^k_{\nu}\right).
$$ 
Now the mapping $K_{12}$ on $\wedge^2H_\Q\otimes H_\Q^{\otimes k}$ is 
the contraction $\wedge^2H_\Q\rightarrow\Q$ times the identity of $H_\Q^{\otimes k}$.
We can now conclude that
$$
K_{12}(V)=K_{12}(q(V))\ \subset \bigoplus_{\text{$\nu$ satisfies (C)}} V^k_{\nu}
$$ as required.
\end{proof}

\vspace{3mm}

\begin{proof}[Proof of Theorem $\ref{thm:bracket}$]
The bracket operation
$$
B: \mathfrak{h}_{g,1}(k)\otimes \mathfrak{h}_{g,1}(l)\ \overset{[\ ,\ ]}{\longrightarrow}\ \mathfrak{h}_{g,1}(k+l)
$$
is {\it not} a morphism of $\mathrm{GL}$-modules so that we have to be careful here.
First we observe that the above bracket can be extended to 
$$
\widetilde{B}: H_\Q^{\otimes (k+2)} \otimes  H_\Q^{\otimes (l+2)}\ \overset{[\ ,\ ]}{\longrightarrow}\  H_\Q^{\otimes (k+l+2)}
$$
which can be described as follows. Here we understand $\mathfrak{h}_{g,1}(k)$ as a submodule of
$H_\Q^{\otimes (k+2)} $ for any $k$.
We identify the domain of $\widetilde{B}$ with
$H_\Q^{\otimes (k+l+4)} $ and define linear isomorphisms
\begin{align*}
f_i^{(k,l)}&:H_\Q^{\otimes (k+l+4)} \rightarrow H_\Q^{\otimes (k+l+4)} \quad (i=2,3,\ldots,l+2)\\
b_i^{(k,l)}&:H_\Q^{\otimes (k+l+4)} \rightarrow H_\Q^{\otimes (k+l+4)} \quad (i=2,3,\ldots,k+2)
\end{align*}
by setting
\begin{align*}
f_i^{(k,l)}&(u_1\otimes u_2\otimes\cdots\otimes u_{k+2}\otimes v_1\otimes v_2\otimes\cdots\otimes v_{l+2})\\
&=(v_1\otimes v_2\otimes\cdots\otimes v_i\otimes u_1\otimes u_2\otimes\cdots\otimes u_{k+2}\otimes v_{i+1}\otimes
\cdots\otimes v_{l+2})\\
b_i^{(k,l)}&(u_1\otimes u_2\otimes\cdots\otimes u_{k+2}\otimes v_1\otimes v_2\otimes\cdots\otimes v_{l+2})\\
&=(u_1\otimes u_2\otimes\cdots\otimes u_i\otimes v_1\otimes v_2\otimes\cdots\otimes v_{l+2}\otimes u_{i+1}\otimes
\cdots\otimes u_{k+2}).
\end{align*}
Observe here that any of these linear isomorphisms corresponds to a certain element
in the symmetric group $\mathfrak{S}_{k+l+4}$ which acts on $H_\Q^{\otimes (k+l+4)}$ 
naturally.
Also define a linear mapping
$$
K_{i}: H_\Q^{\otimes (k+l+4)} \rightarrow H_\Q^{\otimes (k+l+2)} \quad (i=1,2,\ldots, k+l+3)
$$
by setting
$$
K_i(u_1\otimes \cdots\otimes u_{k+l+4})=(u_i\cdot u_{i+1})\, u_1\otimes \cdots\otimes u_{i-1}\otimes 
u_{i+2}\otimes \cdots\otimes u_{k+l+4}.
$$
Then we can write
$$
\widetilde{B}=\sum_{i=2}^{k+2}  K_i\circ b^{k,l}_i - \sum_{i=2}^{\ell+2}  K_i\circ f^{k,l}_i.
$$
It is easy to see that the restriction of $\widetilde{B}$ to $\mathfrak{h}_{g,1}(k)\otimes \mathfrak{h}_{g,1}(l)$
is precisely the bracket operation described in \cite{morita93}.

Now we prove claim $\mathrm{(i)}$. 
Since $\tilde{B}$ is equal to the bracket operation, for any two elements
$\xi\in \widetilde{H}_\lambda$ and $\eta\in \widetilde{H}_\mu$, we have
$$
[\xi,\eta]=\tilde{B}(\xi\otimes\eta).
$$
Since any of the mappings $b^{k,\ell}_i, f^{k,\ell}_i$ corresponds to a certain element
in the symmetric group $\mathfrak{S}_{k+\ell+4}$ as already mentioned above, the image
of $\xi\otimes\eta$ by it is again contained in a $\mathrm{GL}$-submodule
of $H_\Q^{\otimes (k+\ell+4)}$ isomorphic to $\lambda_{\mathrm{GL}}\otimes\mu_{\mathrm{GL}}$.
Now we have to consider various $K_i$ acting on these $\mathrm{GL}$-submodules. Here we 
apply Lemma \ref{lem:K} with replacing $H_\Q^{\otimes (k+2)}$ and $\lambda_{\mathrm{GL}}$ by
$H_\Q^{\otimes (k+\ell+4)}$ and any of the $\mathrm{GL}$-irreducible
component of $\lambda_{\mathrm{GL}}\otimes \mu_{\mathrm{GL}}$, respectively.
Then we can conclude that
$$
\tilde{B}(\xi\otimes\eta)\ \in B(\lambda,\mu)
$$
as required.

Next we prove the condition on $h(\nu)$.
The irreducible decomposition of the
tensor product $\lambda_{\mathrm{GL}}\otimes\mu_{\mathrm{GL}}$ is given by the Littlewood-Richardson rule
(see \cite{fh}) and we see that 
any of the irreducible summands in this decomposition is represented by a
Young diagram whose number of rows, denoted by $h$, satisfies the inequality
$$
\mathrm{max}\, \{h(\lambda),h(\mu)\}\leq h \leq h(\lambda)+h(\mu).
$$
Since the height of the Young diagram $[1^2]$ corresponding to $\wedge^2H_\Q$ is $2$,
we obtain the condition 
$$
\mathrm{max}\, \{h(\lambda),h(\mu)\}-2 \leq h(\nu) \leq h(\lambda)+h(\mu)
$$
on $h(\nu)$ as  required.

The claim $\mathrm{(ii)}$ follows from $\mathrm{(i)}$ because we have the equality
$$
\widetilde{H}_{\lambda^\delta}^{\mathrm{Sp}}=H_{\lambda}.
$$
\end{proof}

\vspace{5mm}

\section{Comparison among $\mathfrak{h}_{g,1}, \mathfrak{h}_{g,*},\mathfrak{h}_{g}$
and a decomposition of $\mathfrak{h}_{g,1}$}\label{sec:comp}

We recall the definitions of three kinds of symplectic derivation algebras, denoted by
$\mathfrak{h}^\Z_{g,1},  \mathfrak{h}^\Z_{g,*}, \mathfrak{h}^\Z_{g}$ from \cite{morita99}.
They correspond to
three kinds of mapping class groups $\mathcal{M}_{g,1}, \mathcal{M}_{g,*}, \mathcal{M}_g$
of $\Sigma_g$, relative to an embedded disk $D^2\subset \Sigma_g$,
relative to a base point $*\in\Sigma_g$ and without any decoration, respectively.
They are Lie algebras defined over $\Z$ and to indicate this fact, we put the superscript
$\Z$ on their symbols. The Lie algebras $\mathfrak{h}_{g,1}, \mathfrak{h}_{g,*},\mathfrak{h}_{g}$
in the subtitle denote the rational forms of them and they are Lie algebras over $\Q$.

Let $\mathcal{L}^\Z_{g,1}=\oplus_{k=1}^\infty \mathcal{L}^\Z_{g,1}(k)$ be the free graded Lie algebra, over $\Z$,
generated by $H$. Then  the degree $k$ part of $\mathfrak{h}^\Z_{g,1}$ can be written as
$$
\mathfrak{h}^\Z_{g,1}(k)=\mathrm{Ker}\, \left(H\otimes \mathcal{L}^\Z_{g,1}(k+1)
\overset{[\ ,\ ]}{\longrightarrow} \mathcal{L}^\Z_{g,1}(k+2)\right).
$$

Next let $\mathcal{L}^\Z_g$ be the graded Lie algebra, over $\Z$, associated to the lower central series of
$\pi_1\Sigma_g$. Then the degree $k$ part of $\mathfrak{h}^\Z_{g,*}$ can be written as
$$
\mathfrak{h}^\Z_{g,*}(k)=\mathrm{Ker}\, \left(H\otimes \mathcal{L}^\Z_{g}(k+1)
\overset{[\ ,\ ]}{\longrightarrow} \mathcal{L}^\Z_{g}(k+2)\right).
$$
Let $\omega_0\in \mathcal{L}_{g,1}(2)$ be the symplectic element and let $I_g=\oplus_{k=2}^\infty I_g(k)$
be the ideal of $\mathcal{L}_{g,1}$ generated by $\omega_0$. Then
a theorem of Labute (see Theorem \ref{thm:Labute} below) says that  
there is a natural isomorphism  $\mathcal{L}_g\cong\mathcal{L}_{g,1}/I_g$.
Keeping this in mind, define an ideal $\mathfrak{j}^\Z_{g,1}$ of $\mathfrak{h}^\Z_{g,1}$
by setting
$$
\mathfrak{j}^\Z_{g,1}(k)=\mathrm{Ker}\, \left(H\otimes I_g(k+1)
\overset{[\ ,\ ]}{\longrightarrow} I_g(k+2)\right).
$$
Then it is easy to see that there is an isomorphism
$$
\mathfrak{h}^\Z_{g,*}\cong\mathfrak{h}^\Z_{g,1}/\mathfrak{j}^\Z_{g,1}.
$$
Finally a result of Asada and Kaneko \cite{ak} implies that 
the injective homomorphism $\pi_1\Sigma_g\rightarrow \mathcal{M}_{g,*}$
induces an injection
$$
\mathcal{L}^\Z_g\subset \mathfrak{h}^\Z_{g,*}
$$
of graded Lie algebras and the former is an ideal of the latter.
Then we define $\mathfrak{h}^\Z_g$ by setting
$$
\mathfrak{h}^\Z_{g}=\mathfrak{h}^\Z_{g,*}/\mathcal{L}^\Z_g.
$$


If we specify the genus $g$, we write $H_g$ instead of $H$. Also fix a symplectic basis
$x_1,y_1,\ldots,x_g,y_g$ for $H_g$ so that we have an injection
$i: H_g\subset H_{g+1}$ and a projection $p: H_{g+1}\rightarrow H_g$ which is induced by
setting $p(x_{g+1})=p(y_{g+1})=0$.
\begin{lem}
Under the natural inclusion $i: H_g^{\otimes (k+2)}\subset H_{g+1}^{\otimes (k+2)}$, we have
$$
i(\mathfrak{h}^\Z_{g,1}(k))\subset \mathfrak{h}^\Z_{g+1,1}(k).
$$
\end{lem}

\begin{proof}
This is clear from the definition.
\end{proof}

\begin{lem}
Under the natural projection $p: H_{g+1}^{\otimes (k+2)}\rightarrow H_{g}^{\otimes (k+2)}$, we have
$$
p(\mathfrak{h}^\Z_{g+1,1}(k))\subset \mathfrak{h}^\Z_{g,1}(k).
$$
\end{lem}

\begin{proof}
Any element $f\in \mathfrak{h}_{g+1,1}(k)$ can be written as
$$
f=\sum_{i=1}^g x_i\otimes \xi_i+x_{g+1}\otimes \xi_{g+1}+
\sum_{i=1}^g y_i\otimes \eta_i+y_{g+1}\otimes \eta_{g+1}
$$
where $\xi_i,\xi_{g+1}, \eta_i, \eta_{g+1}\in \mathcal{L}_{g+1,1}(k+1)$ and
$$
\sum_{i=1}^g [x_i,\xi_i]+[x_{g+1},\xi_{g+1}]+\sum_{i=1}^g [y_i,\eta_i]+[y_{g+1},\eta_{g+1}]=0.
$$
If we apply the projection $p: \mathcal{L}_{g+1,1}(k+2)\rightarrow \mathcal{L}_{g,1}(k+2)$ to the above
equality, we obtain
$$
\sum_{i=1}^g [x_i,p({\xi}_i)]+\sum_{i=1}^g [y_i,p(\eta_i)]=0
$$
because $p: \mathcal{L}_{g+1}\rightarrow \mathcal{L}_{g}$ is a Lie algebra homomorphism.
On the other hand, we have
$$
p(f)=\sum_{i=1}^g x_i\otimes p(\xi_i)+
\sum_{i=1}^g y_i\otimes p(\eta_i).
$$
We can now conclude that $p(f)\in\mathfrak{h}^\Z_{g}(k)$ as required.
\end{proof}

\begin{definition}
Suppose that,  for any $g$, we are given a graded submodule 
$$
m_g=\bigoplus_{k=1}^\infty m_g(k)\subset \mathfrak{h}^\Z_{g,1}=\bigoplus_{k=1}^\infty \mathfrak{h}^\Z_{g,1}(k).
$$

$\mathrm{(i)}\ $ We call $m_g$
{\it $i$-stable} (inclusion stable) if $i(m_g(k))\subset m_{g+1}(k)$ for any $g$ and $k$.

$\mathrm{(ii)}\ $ We call $m_g$
{\it $p$-stable} (projection stable) if $p(m_{g+1}(k))\subset m_{g}(k)$ for any $g$ and $k$.

Even if the graded submodule $m_g$ is non-trivial only for some fixed $k$, namely 
$m_g(l)=0$ for any $l\not=k$ and $g$, we still say that $\{m_g(k)\}_g$ is 
$i$-stable or $p$-stable if it satisfies the above condition.

A similar definition can be applied to the rational form $\mathfrak{h}_{g,1}$ of $\mathfrak{h}^\Z_{g,1}$.
\label{def:ip}
\end{definition}

\begin{prop}

$\mathrm{(i)}\ \text{The Johnson image} \ \mathrm{Im}\, \tau^\Z_{g,1}\ \text{is $i$-stable}$.

$\mathrm{(ii)}\ \text{The ideal}\ \mathfrak{j}^\Z_{g,1}\  \text{is $p$-stable}$.

$\mathrm{(iii)}\ \text{The ${\mathrm{Sp}}$-invariant Lie subalgebra}\ \mathfrak{h}_{g,1}^{\mathrm{Sp}}\  \text{is $p$-stable}$.
\label{prop:ipZ}
\end{prop}

\begin{proof}
Claim $\mathrm{(i)}$ follows from the fact that the following diagram 
$$
\begin{CD}
\mathcal{M}_{g,1}(k) @>{\tau_{g,1}} >> \mathfrak{h}^\Z_{g,1}(k)\\
@V{i}VV @VV{i}V\\
\mathcal{M}_{g+1,1}(k) @>{\tau_{g+1,1}} >> \mathfrak{h}^\Z_{g+1,1}(k)
\end{CD}
$$
is commutative.

Next we prove $\mathrm{(ii)}$. Since $p(\omega_0(g+1))=\omega_0(g)$, the diagram 
$$
\begin{CD}
I_{g+1}(k)  @>{i}>>  \mathcal{L}^\Z_{g+1}(k)\\
@V{p}VV @VV{p}V\\
I_{g}(k) @>{i}>>  \mathcal{L}^\Z_{g}(k)
\end{CD}
$$
is commutative. It follows that the diagram 
$$
\begin{CD}
H\otimes I_{g+1}(k+1)  @>{[\ ,\ ]}>> I_{g+1}(k+2)\\
@V{p}VV @VV{p}V\\
H\otimes I_{g}(k+1)  @>{[\ ,\ ]}>> I_{g}(k+2)
\end{CD}
$$
is also commutative. 
Therefore $p(\mathfrak{j}^\Z_{g+1,1}(k))=\mathfrak{j}^\Z_{g,1}(k)$
as claimed.
\end{proof}

As for the comparison between $\mathcal{M}_{g,1}(k)$ and $\mathcal{M}_{g,*}(k)$,
we have exact sequences
\begin{align*}
0\rightarrow \Z\rightarrow \mathcal{M}_{g,1}(1)&= \mathcal{I}_{g,1}\rightarrow 
\mathcal{M}_{g,*}(1)=\mathcal{I}_{g,*}\rightarrow 1,\\
0\rightarrow \Z\rightarrow \mathcal{M}_{g,1}(2)&= \mathcal{K}_{g,1}\rightarrow 
\mathcal{M}_{g,*}(2)=\mathcal{K}_{g,*}\rightarrow 1
\end{align*}
for $k=1,2$ where the central subgroup $\Z$ is generated by the Dehn twist,
which we denote by $\tau_\partial(g)$, along a simple closed curve which is parallel
to the boundary curve of $\Sigma_{g,1}$.
For $k\geq 3$, it is easy to see that the natural homomorphism
$p: \mathcal{M}_{g,1}(k)\rightarrow\mathcal{M}_{g,*}(k)$ is {\it injective}
so that we can consider $\mathcal{M}_{g,1}(k)$ as a subgroup of $\mathcal{M}_{g,*}(k)$.
Here we have an important result of Hain \cite{hain} using Hodge theory
as follows. After tensoring with $\Q$,
the graded modules associated with the following two filtrations of the Torelli group
$\mathcal{I}_{g,*}=\mathcal{M}_{g,*}(1)$ are isomorphic to each other.
One is the filtration $\{\mathcal{M}_{g,*}(k)\}_{k\geq 2}$ and the other is the
filtration $\{p(\mathcal{M}_{g,1}(k))\}_{k\geq 2}$ where $p: \mathcal{M}_{g,1}(k)\rightarrow \mathcal{M}_{g,*}(k)$
denotes the natural projection. From this result, we can deduce the following 
(this fact was already mentioned in \cite{morita96}).

\begin{prop}
$\mathrm{(i)}\
\mathcal{M}_{g,1}(k)\cong\mathcal{M}_{g,*}(k)\quad \text{for any $k\geq 3$}$.

$\mathrm{(ii)}\ \mathrm{Im}\,\tau_{g,1}(k)\cap \mathfrak{j}_{g,1}(k)=\{0\} \ \text{and}\
\Im\tau_{g,1}^\Z(k)\cong\Im\tau_{g,*}^\Z(k) \quad \text{for any $k\not= 2$}$.
\end{prop}

\begin{proof}
First of all, it is easy to deduce from the result of Hain mentioned above
that,  for any $k\geq 3$ the subgroup
$\mathcal{M}_{g,1}(k)\subset  \mathcal{M}_{g,*}(k)$
has {\it finite} index, because otherwise the associated graded modules 
of the above two filtrations would not be isomorphic. The former part of the claim $\mathrm{(ii)}$
follows from this because if
$\mathrm{Im}\,\tau_{g,1}(k)\cap \mathfrak{j}_{g,1}(k)\not=\{0\}$
for some $k\geq 3$, then it would imply that $\mathcal{M}_{g,1}(k+1)$
will have an {\it infinite} index in $\mathcal{M}_{g,*}(k+1)$.
Note here that $\mathfrak{j}_{g,1}(k)$ is a free $\Z$-module so that it
has no torsion.

Next we prove that $\mathcal{M}_{g,1}(3)= \mathcal{M}_{g,*}(3)$.
For this, we show that $\mathfrak{j}_{g,1}^\Z(2)$ is isomorphic to $\Z$ 
and determine its generator. By definition, we have
$$
\mathfrak{j}_{g,1}^\Z(2)=\mathrm{Ker}\, \left(H\otimes I_g(3)
\overset{[\ ,\ ]}{\longrightarrow} I_g(4)\right).
$$
Since $I_g(2)\cong\Z$ generated by the symplectic class $\omega_0\in \mathcal{L}_{g,1}(2)$,
$I_g(3)=\{[u,\omega_0];u\in H\}\cong H$. It follows that $H\otimes I_g(3)\cong H\otimes H$
and the bracket operation $H\otimes I_g(3)\rightarrow I_g(4)$ is given by
$v\otimes [u,\omega_0]\mapsto [v,[u,\omega_0]]$. It can then be checked that the 
kernel of this bracket operation is isomorphic to $\Z$ generated by the
element
$$
\sum_{i=1}^g \left\{x_i\otimes [y_i,\omega_0]-y_i\otimes [x_i,\omega_0]\right\}
$$
where $x_i, y_i\ (i=1,\ldots,g)$ denotes a symplectic basis of $H$ as before.
On the other hand, a result of \cite{morita89T} (Proposition 1.1) implies that
$
\tau_{g,1}^\Z(2)(\tau_\partial(g))\in \mathfrak{h}^\Z_{g,1}(2)
$
is precisely the above element (up to signs). It follows that
$$
\tau_{g,1}^\Z(2)(\Z)=\mathfrak{j}_{g,1}^\Z(2).
$$
Now let $\varphi\in \mathcal{M}_{g,*}(3)$ be any element and choose any lift $\tilde{\varphi}\in \mathcal{M}_{g,1}(2)$ 
of $\varphi$. Since $\tau^\Z_{g,*}(2)(\varphi)=0$, we can write $\tau^\Z_{g,1}(2)(\tilde\varphi)=m\in\mathfrak{j}^\Z_{g,1}(2)\cong\Z$
for some $m$. Then we have $\tau^\Z_{g,1}(2)(\tilde\varphi\, \tau_\partial(g)^{-m})=0$ so that
$\tilde\varphi\, \tau_\partial(g)^{-m}\in\mathcal{M}_{g,1}(3)$. Since $\tilde\varphi\, \tau_\partial(g)^{-m}$ is also a lift
of $\varphi$, we can conclude that $\mathcal{M}_{g,1}(3)=\mathcal{M}_{g,*}(3)$ as required.
Next we prove $\mathcal{M}_{g,1}(4)=\mathcal{M}_{g,*}(4)$. Let $\psi\in \mathcal{M}_{g,*}(4)$ be any element 
so that $\tau^\Z_{g,*}(3)(\psi)=0$ and choose 
any lift $\tilde{\psi}\in \mathcal{M}_{g,1}(3)$ of $\psi$. Since $\mathrm{Im}\,\tau_{g,1}(3)\cap \mathfrak{j}_{g,1}(3)=\{0\}$
as already proved above, we have $\tau^\Z_{g,1}(3)(\tilde\psi)=0$. Therefore $\tilde\psi\in\mathcal{M}_{g,1}(4)$
and we can conclude $\mathcal{M}_{g,1}(4)=\mathcal{M}_{g,*}(4)$. It is now clear that an inductive argument proves
the claim $\mathrm{(i)}$. The latter part of $\mathrm{(ii)}$ follows from this completing the proof.
\end{proof}

\begin{thm}
Over the rationals, we have a direct sum decomposition
$$
\mathfrak{h}_{g,1}(k)=
\mathfrak{j}_{g,1}(k)\oplus \mathcal{L}_g(k)\oplus \mathrm{Im}\, \tau_g(k)\oplus \mathrm{Cok}\,\tau_g(k).
$$
\end{thm}

\begin{proof}
This follows from the following three exact sequences
\begin{align*}
0\rightarrow \mathfrak{j}_{g,1}\rightarrow &\mathfrak{h}_{g,1}\rightarrow \mathfrak{h}_{g,*}\rightarrow 0,\quad
0\rightarrow \mathcal{L}_{g}\rightarrow \mathfrak{h}_{g,*}\rightarrow \mathfrak{h}_{g}\rightarrow 0,\\
&0\rightarrow \mathrm{Im}\,\tau_g\rightarrow \mathfrak{h}_{g}\rightarrow \mathrm{Cok}\,\tau_g\rightarrow 0.
\end{align*}
\end{proof}

\section{Determination of the ideal $\mathfrak{j}_{g,1}$}\label{sec:j}

In this section, we formulate a method of determining the ${\mathrm{Sp}}$-irreducible decomposition 
of the ideal $\mathfrak{j}_{g,1}$. First we recall the following classical result (see e.g. \cite{r}).

\begin{thm}
Let $F_m$ be a free group of rank $m$ and let 
$$
\mathcal{L}^\Z_{\langle m \rangle}=\bigoplus_{k=1}^\infty \mathcal{L}^\Z_{\langle m \rangle}(k)
$$
be the free graded Lie algebra generated by $H_1(F_m)\cong \Z^m$.
Then the character of $\mathcal{L}_{\langle m \rangle}(k)=\mathcal{L}^\Z_{\langle m \rangle}(k)\otimes\Q$ 
as a $\mathrm{GL}(m,\Q)$-module is given by
$$
\mathrm{ch}(\mathcal{L}_{\langle m \rangle}(k))=\frac{1}{k} \sum_{d|k} \mu(d) (x_1^d+\cdots +x_m^d)^{k/d}.
$$
In particular
$$
\dim\, \mathcal{L}_{\langle m \rangle}(k)=\frac{1}{k} \sum_{d|k} \mu(d) m^{k/d}.
$$
\label{thm:character}
\end{thm}
\noindent
In our context, we have 
$\mathcal{L}^\Z_{\langle 2g \rangle} = \mathcal{L}^\Z_{g,1}$ 
and $\mathcal{L}_{\langle 2g \rangle} = \mathcal{L}_{g,1}$. 

\begin{thm}[Labute \cite{labute}]
Let $\Sigma_g$ be a closed oriented surface of genus $g$ and let 
$\Sigma_{g,1}=\Sigma_g\setminus \mathrm{Int}\, D^2$.
Let $\mathcal{L}^\Z_g$ (resp. $\mathcal{L}^\Z_{g,1}$) denote the
graded Lie algebra associated to the lower central series of
$\pi_1\Sigma_g$ (resp. $\pi_1\Sigma_{g,1}$). 
 Then 
$$
\mathcal{L}^\Z_{g}=\mathcal{L}^\Z_{g,1}/ \langle\omega_0\rangle
$$
where $\langle\omega_0\rangle$ denotes the ideal generated by the symplectic class
$\omega_0\in \mathcal{L}^\Z_{g,1}(2)$.
Furthermore the $k$-th term $\mathcal{L}^\Z_{g}(k)$
of the graded Lie algebra $\mathcal{L}^\Z_g$ is a free $\Z$-module and its
rank is given by
$$
\mathrm{rank}\, \mathcal{L}^\Z_{g}(k)=\frac{1}{k} \sum_{d|k} \mu (k/d) 
\left[ \sum_{0\leq i\leq [d/2]} (-1)^i \frac{d}{d-i} \binom{d-i}{i} (2g)^{d-2i}\right].
$$
\label{thm:Labute}
\end{thm}

In the above formula, the part $i=0$ corresponds to the case of $\mathcal{L}^\Z_{g,1}(k)$.

Based on the above theorem of Labute, we can deduce the following method of determining the 
$\mathrm{Sp}$-irreducible decomposition of 
$\mathfrak{j}_{g,1}(k)$. 
This is because his theorem gives not only the ranks of the relevant modules 
but the structure of them as $\mathrm{GL}$-modules. 
Define
$$
\tilde{I}(k)=- \frac{1}{k} \sum_{d|k} \mu (k/d) 
\left[ \sum_{1\leq i\leq [d/2]} (-1)^i \frac{d}{d-i} \binom{d-i}{i} P_{k/d}^{\otimes (d-2i)}\right]
$$
where $P_j$ denotes the $\mathrm{GL}$-representation corresponding to the power sum
$$
x_1^{j}+y_1^{j}+\cdots +x_g^j+y_g^j
$$
so that $P_1=H_\Q$ and if $d-2i=0$, then set $P_{k/d}^0=[0]$ the $\mathrm{Sp}$-trivial representation.
Then we have the following result.

\begin{thm}
$\mathrm{(i)}\ $
The $\mathrm{Sp}$-irreducible decomposition of $\mathfrak{j}_{g,1}(k)$ is the same as that of the
virtual $\mathrm{GL}$-representation
$
\left(H_\Q\otimes \tilde{I}(k+1)-\tilde{I}(k+2)\right).
$

$\mathrm{(ii)}\ $
The $\mathrm{Sp}$-irreducible decomposition of $\mathcal{L}_{g}(k)$ is the same as that of the
virtual $\mathrm{GL}$-representation
$
\left(\frac{1}{k} \sum_{d|k} \mu(k/d) P_{k/d}^{d}-\tilde{I}(k)\right).
$
\label{thm:jL}
\end{thm}

\begin{remark}
Probably the latter statement of the above theorem has been well-known to specialists for some time
and the former statement is a direct consequence of it. We have made use of these formulae
to compute the Sp-decomposition of $\mathfrak{j}_{g,1}(k)$ mentioned below.
These computations were very important for our study because it enabled us to 
guess the size of 
the ideal $\mathfrak{j}_{g,1}$ inside the whole Lie algebra $\mathfrak{h}_{g,1}$.
We would like to mention that recently an explicit formula for the $\mathrm{Sp}$-character of 
$\mathcal{L}_g$ has been written down by Filip \cite{filip}.
\end{remark}

By using these methods, we have computed the $\mathrm{Sp}$-irreducible decompositions of 
the $\mathrm{Sp}$-modules $\mathfrak{j}_{g,1}(k)$ and $\mathcal{L}_g(k)$ explicitly for all $k\leq 20$. 
By combining with our earlier work in \cite{mss3}, we have now complete information about the
$\mathrm{Sp}$-irreducible decompositions of all the relevant $\mathrm{Sp}$-modules
up to degree $20$. Here we only 
describe the $\mathrm{Sp}$-invariant part. See Tables \ref{tab:j}, \ref{tab:j6}, \ref{tab:j10}.
The first table indicates the dimensions in the stable range whereas in the latter two tables
(the cases $k=6, 10$), 
we describe the unstable computations for later use.

\begin{table}[h]
\caption{$\text{Dimensions of $\mathfrak{h}_{g,1}(k)^{\mathrm{Sp}}, \mathfrak{j}_{g,1}(k)^{\mathrm{Sp}},
\mathfrak{h}_{g,*}(k)^{\mathrm{Sp}},\mathcal{L}_{g}(k)^{\mathrm{Sp}},\mathfrak{h}_{g}(k)^{\mathrm{Sp}}$}$}
\begin{center}
\begin{tabular}{|r|r|r|r|r|r|}
\noalign{\hrule height0.8pt}
\hfil $k$ & $\mathfrak{h}_{g,1}(k)^{\mathrm{Sp}}$ & $\mathfrak{j}_{g,1}(k)^{\mathrm{Sp}}$ & 
$\mathfrak{h}_{g,*}(k)^{\mathrm{Sp}}$ & $\mathcal{L}_{g}(k)^{\mathrm{Sp}}$  & $\mathfrak{h}_{g}(k)^{\mathrm{Sp}}$ \\
\hline
$2$ & $1$& $1$ & 
$0$ & $0$   & $0$  \\
\hline
$4$ & $0$ & $0$& $0$ & $0$ & $0$ \\
\hline
$6$ & $5$ & $2$ & $3$ & $1$ & $2$ \\
\hline
$8$ & $3$ & $1$ &
$2$ & $2$ & $0$  \\
\hline
$10$ & $108$ & $38$  & $70$ & $34$ & $36$ \\
\hline
$12$ & $650$ & $210$ & $440$ & $259$ & $181$ \\     
\hline
$14$ & $8795$ & $2831$ & $5964$ & $3215$ & $2749$ \\    
\hline
$16$ & $110610$ & $34591$ & $76019$ & $41858$ & $34161$ \\    
\hline
$18$ & $1710798$ & $530466$ & $1180332$ & $644758$ & $535574$\\    
\hline
$20$ & $29129790$ & $8980269$ & $20149521$ & $11111008$ & $9038513$ \\                                                 
\noalign{\hrule height0.8pt}
\end{tabular}
\end{center}
\label{tab:j}
\end{table}

\begin{table}[h]
\caption{$\text{Dimensions of $\mathfrak{h}_{g,1}(6)^{\mathrm{Sp}}, \mathfrak{j}_{g,1}(6)^{\mathrm{Sp}},
\mathfrak{h}_{g,*}(6)^{\mathrm{Sp}},\mathcal{L}_{g}(6)^{\mathrm{Sp}},\mathfrak{h}_{g}(6)^{\mathrm{Sp}}$}$}
\begin{center}
\begin{tabular}{|r|r|r|r|r|r|}
\noalign{\hrule height0.8pt}
\hfil $g$ & $\mathfrak{h}_{g,1}(6)^{\mathrm{Sp}}$ & $\mathfrak{j}_{g,1}(6)^{\mathrm{Sp}}$ & 
$\mathfrak{h}_{g,*}(6)^{\mathrm{Sp}}$ & $\mathcal{L}_{g}(6)^{\mathrm{Sp}}$  & $\mathfrak{h}_{g}(6)^{\mathrm{Sp}}$ \\
\hline
$1$ & $1$& $1$ & $0$ & $0$   & $0$  \\
\hline
$2$ & $4$ & $2$& $2$ & $1$ & $1$ \\
\hline
$\geq 3$ & $5$ & $2$ & $3$ & $1$ & $2$ \\                                         
\noalign{\hrule height0.8pt}
\end{tabular}
\end{center}
\label{tab:j6}
\end{table}

\begin{table}[h]
\caption{$\text{Dimensions of $\mathfrak{h}_{g,1}(10)^{\mathrm{Sp}}, \mathfrak{j}_{g,1}(10)^{\mathrm{Sp}},
\mathfrak{h}_{g,*}(10)^{\mathrm{Sp}},\mathcal{L}_{g}(10)^{\mathrm{Sp}},\mathfrak{h}_{g}(10)^{\mathrm{Sp}}$}$}
\begin{center}
\begin{tabular}{|r|r|r|r|r|r|}
\noalign{\hrule height0.8pt}
\hfil $g$ & $\mathfrak{h}_{g,1}(10)^{\mathrm{Sp}}$ & $\mathfrak{j}_{g,1}(10)^{\mathrm{Sp}}$ & 
$\mathfrak{h}_{g,*}(10)^{\mathrm{Sp}}$ & $\mathcal{L}_{g}(10)^{\mathrm{Sp}}$  & $\mathfrak{h}_{g}(10)^{\mathrm{Sp}}$ \\
\hline
$1$ & $3$& $3$ & $0$ & $0$   & $0$  \\
\hline
$2$ & $51$ & $27$& $24$ & $14$ & $10$ \\
\hline
$3$ & $97$ & $37$ & $60$ & $31$ & $29$ \\
\hline
$4$ & $107$ & $38$ & $69$ & $34$ & $35$  \\
\hline
$\geq 5$ & $108$ & $38$  & $70$ & $34$ & $36$ \\                                               
\noalign{\hrule height0.8pt}
\end{tabular}
\end{center}
\label{tab:j10}
\end{table}

\section{Descendants and ancestors}\label{sec:desc}

In this section, we propose a terminology which will hopefully be useful in
analyzing the structure of the Lie algebra $\mathfrak{h}_{g,1}$.

\begin{definition}
Let $V$ be a $\mathrm{GL}$-irreducible component which appears in
$\mathfrak{h}_{g,1}(k)$ so that we can write $V\cong \lambda_{\mathrm{GL}}$
where $\lambda$ is a Young diagram with $(k+2)$ boxes. 
As an $\mathrm{Sp}$-module, $V$ can be decomposed into
a linear combination
$$
V=\bigoplus_{i=0}^ s   V_i
$$
of $\mathrm{Sp}$-irreducible representations $V_i$. Here $V_0$ denotes the 
largest $\mathrm{Sp}$-irreducible component of $V$  so that
$V_0\cong  \lambda_{\mathrm{Sp}}$ and any of the other $V_i\ (i\geq 1)$ is obtained
from $V$ by applying various contractions successively. We say that $V_i$ is an
{\it Sp-descendant} of $V$ of {\it order} $d$ if it is obtained by applying 
contractions $d$ times. In this case $V_i\cong \lambda^{(i)}_{\mathrm{Sp}}$
where $\lambda^{(i)}$ is a Young diagram with $(k+2-2d)$ boxes.
We also call {\it Sp-descendant} of {\it order} $1$ (resp. $2$) an {\it Sp-child}
(resp. an {\it Sp-grandchild}). Conversely we call $V$ the {\it ancestor} of $V_i$
for any $i$.
\label{def:da}
\end{definition}

\begin{example}
Asada-Nakamura \cite{an} proved that there exists a unique copy 
$[2k+1\, 1^2]_{\mathrm{GL}} \subset \mathfrak{h}_{g,1}(2k+1)$ for any $k\geq 1$.
The trace component $[2k+1]_{\mathrm{Sp}} \subset \mathfrak{h}_{g,1}(2k+1)$ is an $Sp$-child of 
this unique copy. 
Enomoto-Satoh \cite{es} proved that there exists a unique copy 
$[2^21^{4k-1}]_{\mathrm{GL}}\subset \mathfrak{h}_{g,1}(4k+1)$ for any $k\geq 1$.
Their {\it anti-trace} component $[1^{4k+1}]_{\mathrm{Sp}}\subset \mathfrak{h}_{g,1}(4k+1)$ 
is an $Sp$-child of this unique copy.
\end{example}

\begin{example}
The $\mathrm{Sp}$-invariant part given in Theorem \ref{thm:ortho} is,
so to speak, the ``last descendants" of each piece $\mathfrak{h}_{g,1}(2k)$.
\end{example}

\begin{thm}
For any $k$ and any $\mathrm{GL}$-irreducible component $V$
of $\mathfrak{h}_{g,1}(k)$ isomorphic to $\lambda_{\mathrm{GL}}$, 
the corresponding  $\mathrm{Sp}$-irreducible component $V_0$ which is isomorphic to
$\lambda_{\mathrm{Sp}}$ is contained in $\mathrm{Im}\, \tau_{g,1}(k)$ in a certain stable range
(we can take $g\geq k+3$).
\label{thm:imtau}
\end{thm}

\begin{remark}
It is a very important problem to determine which $\mathrm{Sp}$-descendants of 
a given $\mathrm{GL}$-irreducible component $V\subset \mathfrak{h}_{g,1}(k)$
belong to $\Im \tau_{g,1}(k)$ and/or remain non-trivial in $H_1(\mathfrak{h}^+_{g,1})$.
\end{remark}

To prove the above theorem, we prepare some terminology.

Let $\sigma_i=(12\cdots i)\in \mathfrak{S}_k$ be the cyclic permutation. Define two elements
\begin{align*}
p_k&=(1-\sigma_k)(1-\sigma_{k-1})\cdots (1-\sigma_2) \\
S_k&=\sum_{j=1}^k \sigma^j_k
\end{align*}
in $\Z[\mathfrak{S}_k]$ both of which act on $H_\Q^{\otimes k}$ linearly. 
It is easy to see that $p_k(H_\Q^{\otimes k})=\mathcal{L}_{g,1}(k)$.

\begin{prop}[see \cite{morita99}]
The subspace $H_\Q\otimes\mathcal{L}_{g,1}(k+1)\subset H_\Q^{\otimes {(k+2)}}$ is invariant under the
action of $S_{k+2}$ and
$$
S_{k+2}\left(H_\Q\otimes\mathcal{L}_{g,1}(k+1)\right)=\mathfrak{h}_{g,1}(k).
$$
It follows that
$$
\mathfrak{h}_{g,1}(k)=S_{k+2}\circ (1\otimes p_{k+1})\left(H_\Q^{\otimes (k+2)}\right).
$$
\label{prop:chah}
\end{prop}

Here we mention a relation with another method of expressing elements of $\mathfrak{h}_{g,1}(k)$,
namely {\it Lie spiders} (see e.g. \cite{levine}). More precisely the following element 
\begin{align*}
S_{k+2}&(u_1\otimes [u_2,[u_3,[\cdots [u_{k+1},u_{k+2}]\cdots ]) \\
=& u_1\otimes [u_2,[u_3,[\cdots [u_{k+1},u_{k+2}]\cdots ]+u_2\otimes [[u_3,[u_4,[\cdots [u_{k+1},u_{k+2}]\cdots ],u_1]+\\
& u_3\otimes [ [u_4,[u_5,[\cdots [u_{k+1},u_{k+2}]\cdots ],[u_1,u_2]]+\cdots +u_{k+2}\otimes 
[\cdots [u_1,u_2],\cdots,u_{k}],u_{k+1}]
\end{align*}
is represented by the Lie spider
\bigskip
\begin{center}
\includegraphics[width=.45\textwidth]{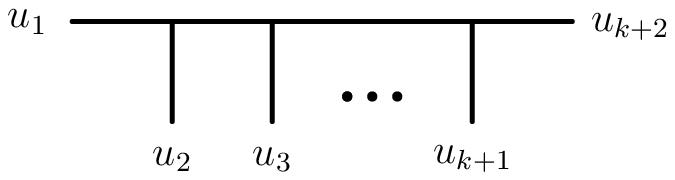}
\end{center}


\noindent
where $u_i\in H_\Q$.

\begin{proof}[Proof of Theorem \ref{thm:imtau}]
We use induction on $k$. It is a classical result of Johnson \cite{johnson}
that $\mathrm{Im}\,\tau_{g,1}(1)=\mathfrak{h}_{g,1}(1)=\wedge^3 H_\Q$.
Hence the claim holds for $k=1$.

It is easy to see that the bracket operation
$$
H_\Q\otimes \mathcal{L}_{g,1}(k)\rightarrow \mathcal{L}_{g,1}(k+1)
$$
is {\it surjective} for any $k\geq 1$. It follows that $\mathfrak{h}_{g,1}(k)$ is generated by the
elements of the form
$$
S_{k+2}(u_1\otimes [u_2,[u_3,[\cdots [u_{k+1},u_{k+2}]\cdots ]).
$$
Let us denote the above element by $\ell(u_1,\ldots,u_{k+2})$.
Then we observe that the highest weight vector for $\lambda_{\mathrm{Sp}}$
can be expressed as a linear combination of elements of the above form where
any of $u_i\ (i=1,\ldots,k+2)$ is equal to some $x_j$ with $j\leq k+2$. Here $x_1,\ldots,x_g,y_1,\ldots,y_g$ is
a symplectic basis of $H$ for $g \geq k+3$. 
This is because the module $\lambda_{\mathrm{Sp}}$, regarded as a subspace of $\lambda_{\mathrm{GL}}$,
is characterized as the intersection of the kernels of all the possible $\mathrm{Sp}$
contractions $\lambda_{\mathrm{GL}}\rightarrow \mu_{\mathrm{Sp}}$ where $\mu_{\mathrm{Sp}}$ denotes
any $\mathrm{Sp}$-module.
Hence it suffices to prove that $\ell(u_1,\ldots,u_{k+2})$
is contained in $\mathrm{Im}\,\tau_{g,1}(k)$ for such $g$ whenever any $u_i$ is equal to 
$x_j \ (j\leq k+2)$.
Now assume $k>1$ and consider the following two elements
$$
\ell(u_1,\ldots,u_{k},x_{k+3}),\quad \ell(y_{k+3},u_{k+1},u_{k+2}).
$$
Then the former element is contained in $\mathrm{Im}\,\tau_{g',1}(k-1)$ with $g' \geq k+3$ 
by the induction assumption
and the latter one belongs to $\mathfrak{h}_{g,1}(1)=\mathrm{Im}\,\tau_{g,1}(1)$ for any $g\geq k+3$.
Also it is easy to see that
$$
[\ell(u_1,\ldots,u_{k},x_{k+3}),\ell(y_{k+3},u_{k+1},u_{k+2})]=\ell(u_1,\ldots,u_{k+2}).
$$
Therefore $\ell(u_1,\ldots,u_{k+2})$ is contained in $\mathrm{Im}\,\tau_{g,1}(k)$ 
with $g \geq k+3$ as required. This completes the proof.
\end{proof}

\begin{prop}

$\mathrm{(i)}\ $ For any $\lambda$ with $|\lambda|=k+2$,
the isotypical component $\widetilde{H}_\lambda\subset \mathfrak{h}_{g,1}(k)$
as well as the corresponding $\mathrm{Sp}$-isotypical component $(\widetilde{H}_\lambda)_0$ is $i$-stable.


$\mathrm{(ii)}\ \text{The $\mathrm{Sp}$-invariant part}\ \mathfrak{h}_{g,1}^{\mathrm{Sp}}\  \text{is p-stable}$.
Furthermore each component of the orthogonal direct sum decomposition of this space given by Theorem $\ref{thm:ortho}$ is $p$-stable.


$\mathrm{(iii)}\ $ $\mathfrak{j}_{g,1}^{\mathrm{Sp}}$ is p-stable. 

$\mathrm{(iv)}\ $ The ideal $[\mathfrak{h}_{g,1},\mathfrak{h}_{g,1}]$ is $i$-stable.
\label{prop:ipQ}
\end{prop}

\begin{proof}
$\mathrm{(i)}$ follows from that the following three facts. The first is that the isotypical component 
$V_\lambda^{k+2}\subset (H_\Q^{\otimes (k+2)})^{\mathrm{Sp}}$ is $i$-stable which follows from the classical 
construction of this component. The second is that the submodule $\mathfrak{h}_{g,1}(k)\subset H_\Q^{\otimes (k+2)}$ 
is $i$-stable. The third is that taking the contractions to any given ``direction" is an $i$-stable operation.
The rest of the assertions follow similarly using the facts that 
$(H_\Q^{\otimes (k+2)})^{\mathrm{Sp}}\subset H_\Q^{\otimes (k+2)}$ is $p$-stable and the bracket operation of
$\mathfrak{h}_{g,1}$ is $i$-stable in an obvious sense. 
\end{proof}

\begin{remark}
In general, 
the $\mathrm{Sp}$-invariant part of the Johnson image $(\mathrm{Im}\, \tau_{g,1})^{\mathrm{Sp}}$ 
is neither $i$-stable nor $p$-stable, 
and $\mathfrak{j}_{g,1}^{\mathrm{Sp}}$ is not $i$-stable. 
These can be checked by direct computation. 
This is one of the reasons why the problem of determining 
(the $\mathrm{Sp}$-invariant part of) the image of the Johnson homomorphism 
as well as the ideal $\mathfrak{j}_{g,1}$ is difficult. 
In a trial to overcome this difficulty, in Section \ref{sec:hsp} 
we introduce two kinds of bases for $\mathfrak{h}_{g,1}(2k)^{\mathrm{Sp}}$. 
\end{remark}

\section{Two kinds of bases for $\mathfrak{h}_{g,1}(2k)^{\mathrm{Sp}}$}\label{sec:hsp}

In this section, we describe a general method of constructing elements 
of $\mathfrak{h}_{g,1}(2k)^{\mathrm{Sp}}$
and by using it we introduce two kinds of bases for it.

Let $\sigma_i=(12\cdots i)\in \mathfrak{S}_{k}$ be the cyclic permutation as before. Define two elements
\begin{align*}
p'_k&=\left(1-(-1)^{k-1}\sigma_k\right)(1-(-1)^{k-2} \sigma_{k-1})\cdots (1+\sigma_2) , \\
S'_k&=\sum_{j=1}^k (-1)^{j(k-1)} \sigma^j_k
\end{align*}
in $\Z[\mathfrak{S}_k]$. Then by combining Proposition \ref{prop:chah} with Proposition \ref{prop:LCD},
we obtain the following result. 

\begin{prop}
$$
\mathfrak{h}_{g,1}(2k)^{\mathrm{Sp}}=
\Phi\left(S'_{2k+2} \circ \sigma_{2k+2} \circ p'_{2k+1} \circ \sigma_{2k+2}^{-1}(\Q\mathcal{D}^\ell (2k+2))\right).
$$
\label{prop:LCDH}
\end{prop}
Thus we obtain a method of constructing elements of $\mathfrak{h}_{g,1}(2k)^{\mathrm{Sp}}$.
We remark here that computation by a computer is much easier in this context of
$\Q\mathcal{D}^\ell (2k+2)$ rather than that of $(H_\Q^{\otimes (2k+2)})^{\mathrm{Sp}}$
because in the latter case the dimension gets very large quickly as the genus grows
while in the former context the computation is independent of the genus.
In particular, the explicit orthogonal decomposition of  $(H_\Q^{\otimes (2k+2)})^{\mathrm{Sp}}$
can be obtained by applying various Young symmetrizers on $\Q\mathcal{D}^\ell (2k+2)$
to obtain the corresponding decomposition of this space and then converting it to the space
$(H_\Q^{\otimes (2k+2)})^{\mathrm{Sp}}$
by applying Proposition \ref{prop:LCD}. To obtain the orthogonal decomposition of 
$\mathfrak{h}_{g,1}(2k)^{\mathrm{Sp}}$, it is enough to apply further the
operator $
S'_{2k+2} \circ \sigma_{2k+2} \circ p'_{2k+1} \circ \sigma_{2k+2}^{-1}
$ to the above decomposition of
$\Q\mathcal{D}^\ell (2k+2)$ and then apply the homomorphism $\Phi$.

More precisely, the explicit procedure goes as follows.
By Theorem \ref{thm:lcdortho}, 
there exists an orthogonal
direct sum decomposition
$$
\Q\mathcal{D}^\ell(2k+2)\cong \bigoplus_{|\lambda |=k+1}
E_{\lambda}
$$
in terms of certain subspaces $E_\lambda$.
As an $\mathfrak{S}_{2k+2}$-submodule of 
$\Q\mathcal{D}^\ell(2k+2)$, $E_\lambda$
is an irreducible $\mathfrak{S}_{2k+2}$-module corresponding to the Young diagram
$
2\lambda.
$
Since the operator 
$S'_{2k+2} \circ \sigma_{2k+2} \circ p'_{2k+1} \circ \sigma_{2k+2}^{-1}$ 
belongs to $\Z[\mathfrak{S}_{2k+2}]$,
if we define
$$
F_\lambda=
(S'_{2k+2} \circ \sigma_{2k+2} \circ p'_{2k+1} \circ \sigma_{2k+2}^{-1})
(E_\lambda),
$$
then $F_\lambda$ is a subspace of $E_\lambda$ and
$$
\Phi(F_\lambda)=H_{\lambda'}\  \subset\ \mathfrak{h}_{g,1}(2k)^{\mathrm{Sp}}.
$$

Thus we obtain a method of constructing elements of $\mathfrak{h}_{g,1}(2k)^{\mathrm{Sp}}$
which respects the orthogonal decomposition and is also independent of the genus $g$.
Indeed, if we choose a basis $\{C_\lambda^{i}; i=1,\ldots,\mathrm{dim}\, F_\lambda\}$ of $F_\lambda$ and set
$v_\lambda^{i}=\Phi(C_\lambda^i)$. Then
\begin{equation}
\{v_\lambda^{i}; |\lambda|=k+1, i=1,\ldots,\mathrm{dim}\, F_\lambda=\mathrm{dim}\, H_{\lambda'}\}
\label{eq:basis}
\end{equation}
is a basis of $\mathfrak{h}_{g,1}(2k)^{\mathrm{Sp}}$ in the stable range. 
We call this a {\it $p$-stable basis} because it is clearly 
$p$-stable in the obvious sense. This basis is suitable for describing 
$\mathfrak{j}_{g,1}^{\mathrm{Sp}}$ which is $p$-stable. 
However, it is not $i$-stable (in fact no basis can be 
$i$-stable) and the description of $(\mathrm{Im}\,\tau_{g,1})^{\mathrm{Sp}}$
is rather complicated. 
In order to analyze this point, define a linear mapping
$$
\mathcal{K}: H_\Q^{\otimes (2k+2)}\rightarrow \Q\mathcal{D}^\ell (2k+2)
$$
by setting
$$
\mathcal{K}(\xi)=\sum_{C\in \mathcal{D}^\ell (2k+2)} \alpha_C(\xi)\, C.
$$
We use the same notation $\mathcal{K}$ for the restriction of the above mapping to the subspace
$\mathfrak{h}_{g,1}(2k)\subset H_\Q^{\otimes (2k+2)}$ as well as
its further restriction to the subspace $\mathfrak{h}_{g,1}(2k)^{\mathrm{Sp}}\subset \mathfrak{h}_{g,1}(2k)$.

\begin{prop}
$\mathrm{(i)\ }$ The linear mapping $\mathcal{K}$ is $i$-stable in the sense that the following
diagram is commutative
$$
\begin{CD}
\mathfrak{h}_{g,1}(2k)  @>{\mathcal{K}}>> \Q\mathcal{D}^\ell (2k+2) \\
@V{$i$}VV @|\\
\mathfrak{h}_{g+1,1}(2k)\ @>{\mathcal{K}}>> \Q\mathcal{D}^\ell (2k+2).
\end{CD}
$$

$\mathrm{(ii)\ }$ The linear mapping
$$
\mathcal{K}:\mathfrak{h}_{g,1}(2k)^{\mathrm{Sp}}
\rightarrow \Q\mathcal{D}^\ell (2k+2)
$$ 
is injective for any $g$. Furthermore
two subspaces $\mathcal{K}(H_\lambda)$ and 
$\mathcal{K}(H_\mu) \ (\lambda\not=\mu)$ are mutually 
orthogonal to each other with respect to the usual Euclidean metric
on $\Q\mathcal{D}^\ell (2k+2)$ which is induced by taking $\mathcal{D}^\ell(2k+2)$
as an orthonormal basis.

$\mathrm{(iii)\ }$
For any element $\xi\in F_\lambda$, we have the equality 
$$
\mathcal{K}(\Phi(\xi))=\mu_{\lambda\,} \xi.
$$
\label{prop:k}
\end{prop}

\begin{proof}
$\mathrm{(i)}$ follows from the fact that the contraction is an $i$-stable operation.
The former part of $\mathrm{(ii)}$ holds because any $\mathrm{Sp}$-invariant tensor 
is detected by some iterated contractions. The latter part follows from a 
stronger statement proved in \cite{morita13}
that two subspaces $\mathcal{K}(U_\lambda)$ and 
$\mathcal{K}(U_\mu) \ (\lambda\not=\mu)$ are mutually 
orthogonal to each other with respect to the usual Euclidean metric.
$\mathrm{(iii)}$ follows similarly because it is proved in the above cited paper
that the equality $\mathcal{K}(\Phi(\xi))=\mu_{\lambda\,} \xi$ holds  for any $\xi\in E_\lambda$.
\end{proof}

\begin{definition}
Let $V\subset\mathfrak{h}_{g,1}(2k)^{\mathrm{Sp}}$ be any specified subspace,
e.g. $(\mathrm{Im}\,\tau_{g,1}(2k))^\mathrm{Sp}$. We call a finite set $D\subset \mathcal{D}^\ell(2k+2)$
a {\it detector} of $V$ if the linear mapping 
\[
 \mathcal{K}_D : \mathfrak{h}_{g,1}(2k)^{\mathrm{Sp}} \xrightarrow{\mathcal{K}} \Q \mathcal{D}^\ell (2k+2) 
\xrightarrow{proj} \Q D
\]
is injective on $V$. 
\end{definition}

Keeping the above Proposition \ref{prop:k} $\mathrm{(iii)}$ in mind, 
we make the following definition.
\begin{definition}
We modify the $p$-stable basis \eqref{eq:basis} by setting
$$
\bar{v}_\lambda^{i}=\frac{1}{\mu_{\lambda'}(g)}\, v_\lambda^{i}
$$
to obtain another basis 
\begin{equation}
\{\bar{v}_\lambda^{i}; |\lambda|=k+1, i=1,\ldots,\mathrm{dim}\, F_\lambda=\mathrm{dim}\, H_{\lambda'}\}
\label{eq:basism}
\end{equation}
of $\mathfrak{h}_{g,1}(2k)^{\mathrm{Sp}}$ in the stable range. We call this 
a {\it normalized} basis.
\end{definition}

By combining the above results, we obtain the following theorem which shows that 
although there is a considerable difference between the
two Lie subalgebras $\mathfrak{j}_{g,1}^{\mathrm{Sp}}$ and $(\mathrm{Im}\,\tau_{g,1})^{\mathrm{Sp}}$
it can be completely analyzed by rescaling each piece in the orthogonal decomposition 
by the corresponding eigenvalue.

\begin{thm}
$\mathrm{(i)}\ $
The subspace $\mathfrak{j}_{g,1}(2k)^{\mathrm{Sp}}\subset \mathfrak{h}_{g,1}(2k)^{\mathrm{Sp}}$ is $p$-stable
so that the description of it in terms of a $p$-stable basis is constant with respect to $g$.

$\mathrm{(ii)}\ $
In contrast with the case $\mathrm{(i)}$ above, the subspace $(\mathrm{Im}\, \tau_{g,1}(2k))^{\mathrm{Sp}}\subset 
\mathfrak{h}_{g,1}(2k)^{\mathrm{Sp}}$ is not $p$-stable. 
However, it is {\it weighted stable} in the following sense. Namely
the description of it in terms of a normalized basis is constant with respect to $g$.
\label{thm:imtauw}
\end{thm}

\begin{proof}
It remains to prove the last claim of $\mathrm{(ii)}$. The set of values under $\mathcal{K}$ of 
$(\mathrm{Im}\, \tau_{g,1}(2k))^{\mathrm{Sp}}$ is constant with respect to $g$ in the stable range.
On the other hand, Proposition \ref{prop:k} $\mathrm{(iii)}$ shows that the value under $\mathcal{K}$ of  
each member of a $p$-stable basis is the corresponding eigenvalue times a constant vector.
Since the normalized basis cancels exactly this factor, the claim holds.
\end{proof}

\section{Description of $\mathfrak{h}_{g,1}(6)^{\mathrm{Sp}}$}\label{sec:w6}

In this section, by applying the general results obtained in the preceding section,
we give a complete description of the $\mathrm{Sp}$-invariant part $\mathfrak{h}_{g,1}(6)^{\mathrm{Sp}}$
of the weight $6$ summand of $\mathfrak{h}_{g,1}$.

We begin by recalling the following decomposition.
\begin{prop}
The $\mathrm{GL}$-irreducible decomposition of $\mathfrak{h}_{g,1}(6)$ in a stable
range is given by
\begin{align*}
\mathfrak{h}_{g,1}(6)&=[62]_{\mathrm{GL}}+[521]_{\mathrm{GL}}+[51^3]_{\mathrm{GL}}+[4^2]_{\mathrm{GL}}+[431]_{\mathrm{GL}}+2[42^2]_{\mathrm{GL}}\\
& +[421^2]_{\mathrm{GL}}+[41^4]_{\mathrm{GL}}+2[3^21^2]_{\mathrm{GL}}+[32^21]_{\mathrm{GL}}+[321^3]_{\mathrm{GL}}+[2^4]_{\mathrm{GL}}+[2^21^4]_{\mathrm{GL}}
\end{align*}
\end{prop}
By combining this with the results of the preceding sections, we can make Table \ref{tab:h268} and Table \ref{tab:j6}
in Section $2$. 
Next we construct an explicit basis of $\mathfrak{h}_{g,1}(6)^{\mathrm{Sp}}\cong\Q^5$
as follows. The set $\mathcal{D}^\ell(8)$ of linear chord diagrams with $8$ vertices has
$105$ elements and we enumerate them by the lexicographic order 
with respect to our notation of linear chord diagrams. 
We apply various Young symmetrizers and the projectors to the space $\Q\mathcal{D}^\ell(8)$. 
Then we obtain $5$ elements
$$
C_i\in \Q\mathcal{D}^\ell (8)\cong\Q^{105}\quad (i=1,2,3,4,5)
$$
explicitly described as 
\begin{align*}
C_1=(&18, -2, -16, -2, 2, 0, -16, 11, 5, 0, 5, -5, 16, -3, -13, -2, -1, 3, 2, 3, -5, \\
& 0, 0, 0, -2, -3, 5, 2, 1, -3, -16, 3, 13, 11, 0, -11, 5, -5, 0, -3, -8, 11, \\
& -2,-3, 5, 0, -5, 5, 5, 0, -5, -5, 0, 5, 5, -5, 0, 0, 0, 0, 16, -11, -5, \\
& -3, 0, 3, -13, 5, 8, 11, 0, -11, 2, 3, -5, 2, -2, 0, 1, -3, 2, -3, 0, 3, \\
&5, 0, -5, -2, -1, 3, -18, 2, 16, 2, -2, 0, 16, -11, -5, 0, -5, 5, -16, 3, 13),
\end{align*}
\begin{align*}
C_2=(&16, 4, -7, 4, -4, 0, -7, -8, -3, 0, -3, -4, 7, 3, 2, 4, -2, -3, -4, -6, 3, \\
& 0, 0, 0, 4, 6, -3, -4, 2, 3, -7, -3, -2, -8, 0, 8, -3, 2, 0, 6, 8, -8,\\
& 4, 6, -3, 0, 3, 4, -3, 0, 3, -4, 0, -2, -3, 2, 0, 0, 0, 0, 7, 8, 3, \\
& 3, 0, -6, 2, -2, -8, -8, 0, 8, -4, -6, 3, -4, 4, 0, 2, 6, -4, 3, 0, -6, \\
&-3, 0, 3, 4, -2, -3, -16, -4, 7, -4, 4, 0, 7, 8, 3, 0, 3, 4, -7, -3, -2),
\end{align*}
\begin{align*}
C_3=(&4, 0, -3, 0, 0, 0, -3, -2, -1, 0, -1, 2, 3, -1, 2, 0, 2, 1, 0, 0, 1, \\
&0, 0, 0, 0, 0, -1, 0, -2, -1, -3, 1, -2, -2, 0, 2, -1, 4, 0, 0, 4, -2, \\
& 0, 0, -1, 0, 1, -2, -1, 0, 1, 2, 0, -4, -1, 4, 0, 0, 0, 0, 3, 2, 1, \\
& -1, 0, 0, 2, -4, -4, -2, 0, 2, 0, 0, 1, 0, 0, 0, -2, 0, 0, -1, 0, 0, \\
& -1, 0, 1, 0, 2, 1, -4, 0, 3, 0, 0, 0, 3, 2, 1, 0, 1, -2, -3, 1, -2),
\end{align*}
\begin{align*}
C_4=(&-2, -4, -2, -4, -2, 0, -2, 1, 1, 0, 1, 5, 2, 1, 5, -4, -3, -1, -2, -3, -1,  \\
& 0, 0, 0, 2, 3, 1, 4, 3, 1, -2, -1, -5, 1, 0, -1, 1, -3, 0, 3, 0, 1, \\
& 2, 3, 1, 0, -1, -5, 1, 0, -1, 5, 0, 3, 1, -3, 0, 0, 0, 0, 2, -1, -1,  \\
& 1, 0, -3, 5, 3, 0, 1, 0, -1, -2, -3, -1, 4, 2, 0, 3, 3, -2, 1, 0, -3,  \\
& 1, 0, -1, -4, -3, -1, 2, 4, 2, 4, 2, 0, 2, -1, -1, 0, -1, -5, -2, -1, -5),
\end{align*}
\begin{align*}
C_5=(&-2, -1, 1, -1, -2, 0, 1, 1, 1, 0, 1, -1, -1, -2, -4, -1, 0, 2, -2, -3, -1,  \\
& 0, 0, 0, 2, 3, 1, 1, 0, -2, 1, 2, 4, 1, 0, -1, 1, 3, 0, 3, 6, 1, \\
&  2, 3, 1, 0, -1, 1, 1, 0, -1, -1, 0, -3, 1, 3, 0, 0, 0, 0, -1, -1, -1,  \\
& -2, 0, -3, -4, -3, -6, 1, 0, -1, -2, -3, -1, 1, 2, 0, 0, 3, -2, -2, 0, -3, \\
& 1, 0, -1, -1, 0, 2, 2, 1, -1, 1, 2, 0, -1, -1, -1, 0, -1, 1, 1, 2, 4).
\end{align*}

Finally we set
$$
v_i=\Phi(C_i)\in \mathfrak{h}_{g,1}(6)^{\mathrm{Sp}}\quad (i=1,2,3,4,5).
$$
Then these elements generate each component of the orthogonal decomposition 
of $\mathfrak{h}_{g,1}(6)^{\mathrm{Sp}}$ as depicted in Table \ref{tab:11} below.

\begin{table}[h]
\caption{$\text{Orthogonal decompositions of $(H_\Q^{\otimes 8})^{\mathrm{Sp}}$ and $\mathfrak{h}_{g,1}(6)^{\mathrm{Sp}}$\  }$}
\begin{center}
\begin{tabular}{|c|c|c|c|c|}
\noalign{\hrule height0.8pt}
\hfil $\lambda$ & $\mu_{\lambda'}\ \text{(eigen value of}\ H_{\lambda})$ 
& \text{$\dim\, U_{\lambda}$} & \text{$\dim\, H_{\lambda}$}  & \text{generators for $H_\lambda$} \\
\hline
$[4]$ & $2g(2g+1)(2g+2)(2g+3)$ & $14$ & $1$ & $v_1$ \\
\hline
$[31]$ & $(2g-2)2g(2g+1)(2g+2)$ & $56$ & $2$ & $v_2, v_3$ \\
\hline
$[2^2]$ & $(2g-2)(2g-1)2g(2g+1)$ & $14$ & $1$ & $v_4$ \\
\hline
$[21^2]$ & $(2g-4)(2g-2)2g(2g+1)$ & $20$ & $1$ & $v_5$ \\
\hline
$[1^4]$ & $(2g-6)(2g-4)(2g-2)2g$ & $1$ & $0$  & {}\\                                                               
\hline 
{} & {}  & $105$ & $5$ & {} \\
\noalign{\hrule height0.8pt}
\end{tabular}
\end{center}
\label{tab:11}
\end{table}

Next we consider the dual elements $\alpha_C\ (C\in \mathcal{D}^\ell(8))$. It turns out that 
the set $D$ of five elements 
$$
(12)(34)(56)(78), (12)(34)(57)(68), (12)(34)(58)(67), (12)(36)(47)(58), (12)(38)(46)(57)
$$
which are the $(1,2,3,8,14)$-th elements in the lexicographic order of $\mathcal{D}^\ell(8)$,
is a detector for $\mathfrak{h}_{g,1}(6)^{\mathrm{Sp}}$.

\begin{prop}
The evaluation homomorphism
$$
\mathcal{K}_D: \mathfrak{h}_{g,1}(6)^{\mathrm{Sp}}\rightarrow \Q^5
$$
given by $\alpha_C$'s corresponding to the above $5$ linear chord diagrams in $\mathcal{D}^\ell(8)$
is an isomorphism for $g\geq 3$. More precisely, we have the following equalities:
\begin{align*}
\mathcal{K}_D(v_1) &= 2g(2g+1) (2g+2) (2g+3) (18, -2, -16, 11, -3), \\
\mathcal{K}_D(v_2)&= (2g-2)2g(2g+1)(2g+2) (16 , 4 , -7 , -8 , 3),  \\
\mathcal{K}_D(v_3)&= (2g-2)2g(2g+1)(2g+2)(4 , 0, -3, -2, -1), \\
\mathcal{K}_D(v_4)&= (2g-2)(2g-1)2g(2g+1) (-2 , -4 , -2 , 1 , 1), \\
\mathcal{K}_D(v_5)&= (2g-4)(2g-2)2g(2g+1)(-2 , -1, 1 , 1 , -2). 
\end{align*}
\end{prop}

The above values were obtained by a computer computation. The result is consistent
with the fact that each element $v_i$ is an eigenvector corresponding to the
prescribed eigenvalue. Also we made computations for several values of $g$ and
found that the answers were the same.

\begin{definition}
In the stable range $g\geq 3$, we set
\begin{align*}
\bar{v}_1 &= \frac{1}{2g(2g+1) (2g+2) (2g+3)} v_1 \\
\bar{v}_2 &= \frac{1}{(2g-2)2g(2g+1)(2g+2)} v_2 \\
\bar{v}_3 &= \frac{1}{(2g-2)2g(2g+1)(2g+2)} v_3 \\
\bar{v}_4 &=\frac{1}{(2g-2)(2g-1)2g(2g+1)} v_4 \\
\bar{v}_5 &= \frac{1}{(2g-4)(2g-2)2g(2g+1)} v_5
\end{align*}
so that $\{\bar{v}_i\}_i$ is a {\it normalized} basis of  $\mathfrak{h}_{g,1}(6)^{\mathrm{Sp}}$.
\end{definition}

Here let us recall the Enomoto-Satoh map introduced in \cite{es}, 
which is a linear map $ES_k: \mathfrak{h}_{g,1}(k)\rightarrow\mathfrak{a}_g(k-2)$ where
$$
\mathfrak{a}_g=\bigoplus_{k=0}^\infty \mathfrak{a}_g(k)
$$
denotes the associative version of the derivation Lie algebras defined by 
Kontsevich \cite{kontsevich1}\cite{kontsevich2}. They proved that
$\mathrm{Im}\,\tau_{g,1}(k)$ is contained in $\mathrm{Ker}\,ES_k$ and it is
a very important problem to study the quotient $\mathrm{Ker}\,ES_k/\mathrm{Im}\,\tau_{g,1}(k)$.
We know from Theorem \ref{thm:t6} that this quotient is trivial for all $k\leq 5$.

Let us define the {\it normalizer} of the Johnson image $\mathrm{Im}\,\tau_{g,1}$ in 
$\mathfrak{h}_{g,1}$ as $\mathcal{N}=\oplus_k \mathcal{N}(k)$ where
$$
\mathcal{N}(k)=\{\varphi\in \mathfrak{h}_{g,1}(k); \text{$[\varphi,\psi]\in \mathrm{Im}\,\tau_{g,1}(k+l)$
for any $l$ and $\psi\in \mathrm{Im}\,\tau_{g,1}(l)$}\}.
$$
This should be important in the study of the arithmetic mapping class group because
the Galois obstructions appear in $\mathfrak{h}_{g,1}^{\mathrm{Sp}}$
and they normalize the image of the Johnson homomorphisms (cf. \cite{matsumotopcmi}).

\begin{prop}
An element $\varphi\in \mathfrak{h}_{g,1}(k)$ belongs to $\mathcal{N}(k)$ if and only if
$[\varphi,\mathfrak{h}_{g,1}(1)]\subset \mathrm{Im}\,\tau_{g,1}(k+1)$.
\label{prop:galois}
\end{prop}

\begin{proof}
Recall that Hain \cite{hain} proved that $\mathrm{Im}\,\tau_{g,1}$ is generated by the 
degree $1$ part.
It follows that
any element of $\mathrm{Im}\,\tau_{g,1}(k)$ with $k\geq 2$ can be described as 
a linear combination of brackets of two elements in $\mathrm{Im}\,\tau_{g,1}$ with lower degrees.
Now consider the Jacobi identity
$$
[\varphi,[\psi,\chi]]+[\psi,[\chi,\varphi]]+[\chi,[\varphi,\psi]]=0.
$$
If we assume $\psi,\chi\in\mathrm{Im}\,\tau_{g,1}$ and $\varphi$ normalizes both of $\psi$ and $\chi$,
then the above identity implies that $\varphi$ also normalizes the bracket $[\psi,\chi]$.
The claim follows from an easy inductive argument using this fact. A similar argument shows that
the bracket mapping
$$
\mathrm{Im}\,\tau_{g,1}(k)\otimes \mathrm{Im}\,\tau_{g,1}(1)\overset{[\ ,\ ]}{\longrightarrow}
\mathrm{Im}\,\tau_{g,1}(k+1)
$$
is {\it surjective} for any $k$.
\end{proof}

\begin{thm}
$\mathrm{(i)}\ \text{$(\mathrm{Im}\,\tau_{g,1}(6))^{\mathrm{Sp}}$ is spanned by the following $2$ elements}$
$$
\tau_1=3 \bar{v}_1-\bar{v}_2+8\bar{v}_3+5\bar{v}_5,\quad
\tau_2=6\bar{v}_1+3\bar{v}_2+36\bar{v}_3+25\bar{v}_4-25\bar{v}_5.
$$
$\mathrm{(ii)}\ \text{$\mathfrak{j}_{g,1}(6)^{\mathrm{Sp}}$ is spanned by the following $2$ elements}$
$$
j_1=v_2-5v_3-4v_4+2v_5,\quad
j_2=3v_1+3v_2+3v_3-v_4-2v_5.
$$
$\mathrm{(iii)}\ \text{$\mathrm{dim}\, (\mathrm{Ker}\, ES_6/\mathrm{Im}\,\tau_{g,1}(6))^{\mathrm{Sp}}=1.$}$

\noindent
$\mathrm{(iv)}$ Modulo $(\mathrm{Im}\,\tau_{g,1}(6))^{\mathrm{Sp}}$, there exists a unique element
in $\mathfrak{h}_{g,1}(6)^{\mathrm{Sp}}$ which normalizes $\mathrm{Im}\,\tau_{g,1}$.
Namely
$$
\mathrm{dim}\, (\mathcal{N}(6)/\mathrm{Im}\,\tau_{g,1}(6))^{\mathrm{Sp}}=1.
$$
\label{thm:h6sp}
\end{thm}

\begin{proof}
First we prove $\mathrm{(i)}$. Asada and Nakamura \cite{an} proved that the 
leading term $[31^2]_{\mathrm{Sp}}\subset \mathfrak{h}_{g,1}(3)$ is 
included in the image of the Johnson homomorphism. Let $\xi$ be the
highest weight vector of this summand. We set $\eta=\iota(\xi)$ where
$\iota$ denotes the symplectic automorphism given by $x_i\mapsto y_i, y_i\mapsto -x_i$ 
for all $i$. 
Then we have $[\xi,\eta]\in \mathrm{Im}\, \tau_{g,1}(6)$ and explicit computation shows that
$$
\mathcal{K}_D([\xi,\eta])=(0, 0, 0, -2, 2).
$$
Next, by using the original work of Johnson \cite{johnson} that
$\mathrm{Im}\, \tau_{g,1}(1)=\mathfrak{h}_{g,1}(1)\cong\wedge^3 H_\Q$,
we set
$$
\phi=[y_1\wedge y_2\wedge y_3,[y_1\wedge y_2\wedge y_3,\psi]]
$$
where
$$
\psi=[[x_1\wedge x_2\wedge x_5,x_3\wedge x_4\wedge y_5],[x_1\wedge x_2\wedge x_6,x_3\wedge x_4\wedge y_6]].
$$
Then clearly $\phi\in \mathrm{Im}\, \tau_{g,1}(6)$ and explicit computation shows that
$$
\mathcal{K}_D(\phi)=(4, -1, -4, 4, -4).
$$
It follows that $\mathcal{K}_D(\mathrm{Im}\,\tau_{g,1}(6)^{\mathrm{Sp}})$ is spanned by the $2$ vectors
$$
t_1=(0, 0, 0, -2, 2),\ t_2= (4, -1, -4, 4, -4)
$$
independent of the genus $g$ because $\mathrm{Im}\,\tau_{g,1}$ is $i$-stable.
Since we already know the values of $\mathcal{K}_D$ on each member of our
normalized basis, we can determine the description of 
$(\mathrm{Im}\,\tau_{g,1}(6))^{\mathrm{Sp}}$ as follows. We have the following 
equalities
\begin{align*}
\mathcal{K}_D(9\bar{v}_1-8\bar{v}_2+4\bar{v}_3-25\bar{v}_4+50\bar{v}_5)+90\, t_1&=0,\\
\mathcal{K}_D(-27\bar{v}_1+14\bar{v}_2-52\bar{v}_3+25\bar{v}_4-80\bar{v}_5)+90\, t_2&=0.
\end{align*}
It follows that $(\mathrm{Im}\,\tau_{g,1}(6))^{\mathrm{Sp}}$ is spanned by the following $2$ elements
$$
9\bar{v}_1-8\bar{v}_2+4\bar{v}_3-25\bar{v}_4+50\bar{v}_5,\quad
-27\bar{v}_1+14\bar{v}_2-52\bar{v}_3+25\bar{v}_4-80\bar{v}_5.
$$
The claimed $2$ elements $j_1, j_2$ are obtained by a change of basis.

Next we prove $\mathrm{(ii)}$. We consider the following two elements
$$
p_1=S_8(x_1\otimes [\omega_0,[y_2,[x_2,[\omega_0,y_1]]]]),\quad
p_2=S_8(x_1\otimes [y_2,[\omega_0,[x_2,[\omega_0,y_1]]]]).
$$
It is easy to see that both of these elements are contained in $\mathfrak{j}_{g,1}(6)$ for {\it any} $g$
because there are two copies of the symplectic form $\omega_0$ in the corresponding Lie
spiders. Since we already know that $\mathfrak{j}_{g,1}(6)^\mathrm{Sp}$ is
$p$-stable, to determine the required coefficients we may 
choose the genus $g$ to be any one in the stable range $g\geq 3$.
So we set $g=3$ and begin computation. Then explicit computation shows that
$$
\mathcal{K}_D(p_1)=4\, (98,-7,-81,23,-10),\quad
\mathcal{K}_D(p_2)=24\, (0,-3,-3,0,-1).
$$
Our task is then to find two linear combinations of $v_i$'s whose
values under $\mathcal{K}_D$ generate the $2$ dimensional space
spanned by the above two $5$ dimensional vectors.
Keeping $g=3$, we obtain the following two linear relations
\begin{align*}
\mathcal{K}_D(-36 v_1-23 v_2-101 v_3-40 v_4+50 v_5)+30240\, (98,-7,-81,23,-10)&=0,\\
\mathcal{K}_D( v_2-5 v_3-4 v_4+2 v_5)+6048\, (0,-3,-3,0,-1)&=0.
\end{align*}
Since we have an identity
\begin{align*}
& -36 v_1-23 v_2-101 v_3-40 v_4+50 v_5\\
&=13(v_2-5 v_3-4 v_4+2 v_5)-12(3v_1+3v_2+3v_3-v_4-2v_5)\\
&=13\, j_1-12\, j_2,
\end{align*}
we are done.

To verify the accuracy of the above result, we made similar computations for the cases $g=4,5,6$ and checked that the answers
were the same as above.

Next we prove $\mathrm{(iii)}$. Explicit computation shows that $\mathfrak{a}_g(4)^{\mathrm{Sp}}\cong\Q^2$
and $ES_6(j_1), ES_6(j_2)\in \mathfrak{a}_g(4)^{\mathrm{Sp}}$ are linearly independent. Therefore
$\mathrm{dim}\, (\mathrm{Ker}\, ES_6)^{\mathrm{Sp}}=3$. Since we already know that 
$\mathrm{dim}\, (\mathrm{Im}\, \tau_{g,1}(6))^{\mathrm{Sp}}=2$, the claim follows.
We checked that $ES_6(\tau_1)=ES_6(\tau_2)=0$ as it should be.

Finally we prove $\mathrm{(iv)}$. 
By combining the method of \cite{mss3} and Theorem \ref{thm:jL}, we have
\begin{align*}
\mathfrak{h}_{g,1}(7)&= 10 [1]_{\mathrm{Sp}}\oplus 15 [1^3]_{\mathrm{Sp}}\oplus \text{other terms},\\
\mathfrak{j}_{g,1}(7)&=3 [1]_{\mathrm{Sp}}\oplus 3 [1^3]_{\mathrm{Sp}}\oplus \text{other terms}.
\end{align*}
Also by computing the bracket mapping
$$
\mathfrak{h}_{g,1}(1)\otimes \mathrm{Im}\,\tau_{g,1}(6)\rightarrow \mathfrak{h}_{g,1}(7)
$$
explicitly, we find $6 [1]_{\mathrm{Sp}}\oplus 12 [1^3]_{\mathrm{Sp}}$ sitting inside
$\mathrm{Im}\,\tau_{g,1}(7)$. At this stage, we can conclude that the $[1^3]_{\mathrm{Sp}}$-isotypical
component of $\mathrm{Im}\,\tau_{g,1}(7)$ is $12 [1^3]_{\mathrm{Sp}}$ and the multiplicity of the
$[1]_{\mathrm{Sp}}$-isotypical
component of $\mathrm{Im}\,\tau_{g,1}(7)$ is $6$ or $7$.
Then we compute the Enomoto-Satoh map
$$
ES_7: \mathfrak{h}_{g,1}(7)\rightarrow \mathfrak{a}_g(5)=15 [1]_{\mathrm{Sp}}\oplus 15 [1^3]_{\mathrm{Sp}}\oplus \text{other terms}
$$
and find that it hits $3 [1]_{\mathrm{Sp}}\oplus 2 [1^3]_{\mathrm{Sp}}\subset \mathfrak{a}_g(5)$ while
it detects only $2[1]_{\mathrm{Sp}}$ out of $3[1]_{\mathrm{Sp}}\subset \mathfrak{j}_{g,1}(7)$.
We can now conclude that the $[1]_{\mathrm{Sp}}$-isotypical
component of $\mathrm{Im}\,\tau_{g,1}(7)$ is $6 [1]_{\mathrm{Sp}}$ and
we obtain the following exact sequence
$$
0\rightarrow(\mathrm{Im}\,\tau_{g,1}(7))_1\cong 6 [1]_{\mathrm{Sp}}\oplus 12 [1^3]_{\mathrm{Sp}}
\rightarrow
(\mathrm{Ker}\,ES_7)_1 \overset{(r,s)}{\longrightarrow}  [1]_{\mathrm{Sp}}\oplus  [1^3]_{\mathrm{Sp}}
\rightarrow 0
$$
where 
$(\cdot)_1$ denotes 
the projection onto 
the $([1]_{\mathrm{Sp}}\oplus [1^3]_{\mathrm{Sp}})$-isotypical components and 
$r, s$ denote certain homomorphisms which we constructed explicitly.
Next we consider the homomorphism
$$
\mathcal{B}_1:\mathfrak{h}_{g,1}(6)^{\mathrm{Sp}}\cong\Q^5\rightarrow
\mathrm{Hom} \left(\mathfrak{h}_{g,1}(1), \mathfrak{a}_g(5)\right)^{\mathrm{Sp}}
$$
defined by $\mathcal{B}_1(v)(\xi)=ES_7([v,\xi])\ (v\in \mathfrak{h}_{g,1}(6)^{\mathrm{Sp}}, \xi\in \mathfrak{h}_{g,1}(1))$.
Computation shows that $\mathcal{B}_1$ hits only $1$ dimension in the target so that 
$\mathrm{dim}\, \mathrm{Ker}\, \mathcal{B}_1=4$.
Finally we consider the homomorphism
$$
\mathcal{B}_2: \mathrm{Ker}\, \mathcal{B}_1\subset \mathfrak{h}_{g,1}(6)^{\mathrm{Sp}}\rightarrow
\mathrm{Hom} \left(\mathfrak{h}_{g,1}(1),  [1]_{\mathrm{Sp}}\oplus  [1^3]_{\mathrm{Sp}}\right)^{\mathrm{Sp}}
\cong\Q^2
$$
defined by $\mathcal{B}_2(v)(\xi)=(r,s)([v,\xi])$. It turns out that $\mathcal{B}_2$ hits also
only $1$ dimension in the target so that
$\mathrm{dim}\, \mathrm{Ker}\, \mathcal{B}_2=3$. Thus besides $(\mathrm{Im}\, \tau_{g,1}(6))^{\mathrm{Sp}}$,
which is $2$ dimensional, there exists one more dimension which normalizes $\mathfrak{h}_{g,1}(1)$.
In view of Proposition \ref{prop:galois}, the claim is now proved.
\end{proof}

\begin{remark}
It would be worthwhile to investigate whether the above summands $[1]_{\mathrm{Sp}}\oplus  [1^3]_{\mathrm{Sp}}$
which the Enomoto-Satoh map does not detect, are detected by Conant's new trace map given in
\cite{conant} or not.
\end{remark}

\section{Structure of the genus $1$ case $\mathfrak{h}_{1,1}$}\label{sec:g1}

In this section, we consider the case of genus $1$ motivated by a recent work of Hain and Matsumoto
mentioned in the introduction. 

\begin{prop}
$\mathrm{(i)\ }$
The irreducible decomposition of $[k l ]_{\mathrm{GL}}$
as an $\mathrm{Sp}$-module is given by
$$
[kl]_{\mathrm{GL}}=[kl ]_{\mathrm{Sp}}+[k-1,l-1 ]_{\mathrm{Sp}}+\cdots+ [k-l+1, 1]_{\mathrm{Sp}}+[k-l ]_{\mathrm{Sp}}.
$$

$\mathrm{(ii)\ }$
The restriction of $[k\ell ]_{\mathrm{GL(2,\Q)}}$
to the subgroup $\mathrm{Sp}(2,\Q)=\mathrm{SL}(2,\Q)\subset \mathrm{GL}(2,\Q)$
is given by
$$
\mathrm{Res}^{\mathrm{GL}(2,\Q)}_{\mathrm{SL}(2,\Q)}\, [kl]_{\mathrm{GL(2,\Q)}}=
[k-l]_{\mathrm{SL}(2,\Q)}.
$$
\label{prop:kl}
\end{prop}

\begin{proof}
$\mathrm{(i)\ }$ follows from the restriction law of the pair $\mathrm{Sp}(2g,\Q)\subset\mathrm{GL}(2g,\Q)$
and $\mathrm{(ii)\ }$ follows from a further restriction to the subgroup 
$\mathrm{Sp}(2,\Q)=\mathrm{SL}(2,\Q)\subset \mathrm{GL}(2,\Q)$
(cf. \cite{fh}).
\end{proof}

Let $\mathfrak{h}_{g,1}(k)_{h\leq 2}$ denote the submodule of $\mathfrak{h}_{g,1}(k)$
consisting of $\mathrm{GL}$-isotypical components of type $\lambda_{\mathrm{GL}}$
with $h(\lambda)\leq 2$. Then as is well known, $\mathfrak{h}_{g,1}(k)_{h\leq 2}$ together
with Proposition \ref{prop:kl} completely determines the $\mathrm{Sp}(2,\Q)$-irreducible
decomposition of $\mathfrak{h}_{1,1}(k)$. In this case, we can see that the converse is also
true. Namely the $\mathrm{Sp}(2,\Q)$-irreducible
decomposition of $\mathfrak{h}_{1,1}(k)$ determines $\mathfrak{h}_{g,1}(k)_{h\leq 2}$
completely.
We indicate in Table \ref{tab:hg=1} the explicit decompositions of the two modules
for $k\leq 18$.

\begin{table}[h]
\caption{$\text{Irreducible decompositions of $\mathfrak{h}_{1,1}(k)$ and $\mathfrak{h}_{g,1}(k)_{h\leq 2}$}$}
\begin{center}
\begin{tabular}{|c|l|l|}
\noalign{\hrule height0.8pt}
\hfil $k$& $\hspace{1cm}\text{irreducible components of $\mathfrak{h}_{1,1}(k)$}$ 
&  $\hspace{1cm}\text{irreducible components of $\mathfrak{h}_{g,1}(k)_{h\leq 2}$}$ \\
\hline
$1$ & $\{0\}$  & $\{0\}$ \\
\hline
$2$ & $[0]$ & $[2^2]$\\
\hline
$3$ & $\{0\}$ & $\{0\}$\\
\hline
$4$ & $[2]$ & $[42]$\\
\hline
$5$ & $\{0\}$ & $\{0\}$\\
\hline
$6$ & $[4][0]$ & $[62][44]$\\
\hline
$7$ & $[3]$ & $[63]$\\
\hline
$8$ & $[6]2[2]$ & $[82]2[64]$\\
\hline
$9$ & $[5][3][1]$ & $[83][74][65]$\\
\hline
$10$ & $[8][6]3[4][2]3[0]$ & $[10\ 2][93]3[84][75]3[66]$\\
\hline
$11$ & $[7]2[5]4[3]2[1]$  & $[10\ 3]2[94]4[85]2[76]$\\
\hline
$12$ & $[10][8]5[6]4[4]8[2]$ & $[12\ 2][11\ 3]5[10\ 4]4[95]8[86]$\\
\hline
$13$ & $2[9]3[7]8[5]9[3]6[1]$ & $2[12\ 3]3[11\ 4]8[10\ 5]9[96]6[87]$\\ 
\hline
$14$  & $[12][10]7[8]9[6]18[4]$ & $[14\ 2][13\ 3]7[12\ 4]9[11\ 5]18[10\ 6]$
\\ {}& $11[2]11[0]$ & $11[97]11[8^2]$\\ 
\hline
$15$ & $2[11]5[9]14[7]21[5]$ & $2[14\ 3]5[13\ 4]14[12\ 5]21[11\ 6]$
\\ {}& $26[3]17[1]$ & $26[10\ 7]17[98]$\\ 
\hline
$16$ & $[14]2[12]9[10]16[8]$ &$[16\ 2]2[15\ 3]9[14\ 4]16[13\ 5]38[12\ 6]$
\\ {}& $38[6]38[4]46[2]10[0]$ & $38[11\ 7]46[10\ 8]10[9^2]$\\ 
\hline
$17$ & $2[13]7[11]23[9]42[7]$  & $2[16\ 3]7[15\ 4]23[14\ 5]42[13\ 6]$
\\ {}& $68[5]72[3]48[1]$ & $68[12\ 7]72[11\ 8]48[10\ 9]$\\ 
\hline
$18$ & $[16]2[14]12[12]26[10]67[8]$ & $[18\ 2]2[17\ 3]12[16\ 4]26[15\ 5]67[14\ 6]$
\\ {}& $96[6]138[4]100[2]57[0]$ & $96[13\ 7]138[12\ 8]100[11\ 9]57[10^2]$\\   
\noalign{\hrule height0.8pt}
\end{tabular}
\end{center}
\label{tab:hg=1}
\end{table}

\begin{prop}
$\mathrm{(i)}\ \text{
The leading term in the irreducible decomposition of $\mathfrak{h}_{1,1}(k)$
is given by}$
\begin{align*}
& [k-2]\quad (\text{$k$ is even}),\\
&\lfloor k/6\rfloor [k-4]\quad (\text{$k$ is odd $\geq 7$}) , \ \{0\} \ \text{for $k=1,3,5$}.
\end{align*}

$\mathrm{(ii)}\ \text{There exists a natural isomorphism}$
$$
\mathfrak{h}_{1,1}(2k)^{\mathrm{Sp}}\cong H_{[k+1]}
$$
which is a special case of the decomposition of $\mathfrak{h}_{g,1}(2k)^{\mathrm{Sp}}$ given by Theorem $\ref{thm:ortho}$.
\label{prop:g1}
\end{prop}

\begin{proof}
A proof of the former case of $\mathrm{(i)}$, which should be well known, can be given as follows.
Asada and Nakamura \cite{an} proved that the leading term of $\mathfrak{h}_{g,1}(2k)$
is $[2k, 2]_{\mathrm{GL}}$. By Proposition \ref{prop:kl}, this summand yields $[2k-2]_{\mathrm{Sp}}$
as its unique grandchild. Its further restriction to the genus $1$ case, namely the restriction to 
the subgroup $\mathrm{Sp}(2,\Q)\subset \mathrm{Sp}(2g,\Q)$ is the required summand
$[2k-2]_{\mathrm{Sp}(2,\Q)}$. Again by Proposition \ref{prop:kl}, it is easy to see that no other
summand in $\mathfrak{h}_{g,1}(2k)$ can yield this component implying that the multiplicity of this
component is $1$.

Next we prove the latter part of $\mathrm{(i)}$.
As an $\mathrm{Sp} (2,\mathbb{Q})=\mathrm{SL}(2,\mathbb{Q})$-representation, 
the character of $\mathfrak{h}_{1,1} (k)$ is 
\begin{align*}
\mathrm{ch}(\mathfrak{h}_{1,1}(k))&=\mathrm{ch}(H_\mathbb{Q})
\mathrm{ch}(\mathcal{L}_{1,1}(k+1))-\mathrm{ch}(\mathcal{L}_{1,1}(k+2))\\
&=(x_1+x_1^{-1})\mathrm{ch}(\mathcal{L}_{1,1}(k+1))-
\mathrm{ch}(\mathcal{L}_{1,1}(k+2)),
\end{align*}
\noindent
where  the character of $\mathcal{L}_{1,1}(l)$ is given by
\[\mathrm{ch}(\mathcal{L}_{1,1}(l))=\frac{1}{l} \sum_{d|l} \mu(d) (x_1^d+x_1^{-d})^{l/d}
=\frac{1}{l} \sum_{d|l} \mu(d) 
\left\{\sum_{i=0}^{l/d} \begin{pmatrix} l/d \\ i \end{pmatrix} x_1^{l-2di} \right\}\] 
(cf. Theorem \ref{thm:character}).
Then the claim follows by checking the leading term  with respect to $x_1$ of this Laurent polynomial. 

From the above formula, the degree of $\mathrm{ch}(\mathcal{L}_{1,1}(l))$ is 
at most $l$. However, the well-known formula 
$\sum_{d|l} \mu (d)=0$ for $l \ge 2$ says that 
the coefficient of $x_1^l$ is $0$. The next term appears in degree $l-2$, 
because the degree of $x_1$ in $\mathrm{ch}(\mathcal{L}_{1,1}(l))$ decreases by two each time. 
Since we have $d=i=1$ in this case, the coefficient of $x_1^{l-2}$ is $1$. 

To see the coefficient of $x_1^{l-4}$, consider the terms with $di=2$. 
If $l$ is odd, it suffices only to see the term with $d=1$ and $i=2$. 
The contribution of this term to the coefficient of 
$x_1^{l-4}$ is $\displaystyle\frac{\mu (1)}{l}\begin{pmatrix} l \\ 2 \end{pmatrix}=
\displaystyle\frac{l-1}{2}$. 
If $l$ is even, another term with $d=2$ and $i=1$ is involved. 
The contribution is given by 
$\displaystyle\frac{\mu (2)}{l}\begin{pmatrix} l/2 \\ 1 \end{pmatrix}=
-\displaystyle\frac{1}{2}$. 
Therefore the coefficient of $x_1^{l-4}$ is $\left\lfloor \displaystyle\frac{l-1}{2}\right\rfloor$ for 
$l$ of both cases. 

To see the coefficient of $x_1^{l-6}$, consider the terms with $di=3$. 
If $l \not\equiv 0$ modulo $3$, it is enough to see only the term with $d=1$ and $i=3$. 
The contribution of this term to the coefficient of 
$x_1^{l-6}$ is $\displaystyle\frac{\mu (1)}{l}\begin{pmatrix} l \\ 3 \end{pmatrix}=
\displaystyle\frac{(l-1)(l-2)}{6}$. 
Otherwise, we have another term with $d=3$ and $i=1$. 
The contribution is 
$\displaystyle\frac{\mu (3)}{l}\begin{pmatrix} l/3 \\ 1 \end{pmatrix}=
-\displaystyle\frac{1}{3}$. 
Therefore the coefficient of $x_1^{l-6}$ is 
$\left\lfloor \displaystyle\frac{(l-1)(l-2)}{6}\right\rfloor$.  

Consequently, we have 
\[\mathrm{ch}(\mathcal{L}_{1,1}(l))=x_1^{l-2}+
\left\lfloor \frac{l-1}{2}\right\rfloor x_1^{l-4}+\left\lfloor \frac{(l-1)(l-2)}{6}\right\rfloor x_1^{l-6}+
(\text{lower degree terms})\]
for $l \ge 2$. Hence 
\begin{align*}
& \ \mathrm{ch}(\mathfrak{h}_{1,1}(k))\\
=&\ (x_1+x_1^{-1})
\left(x_1^{k-1}+\left\lfloor \frac{k}{2}\right\rfloor x_1^{k-3}
+\left\lfloor \frac{k(k-1)}{6}\right\rfloor x_1^{k-5}+\cdots\right)\\
&\ -\left(x_1^k+\left\lfloor \frac{k+1}{2}\right\rfloor x_1^{k-2}
+\left\lfloor \frac{(k+1)k}{6}\right\rfloor x_1^{k-4}+\cdots\right)\\
=&\ \left(1+\left\lfloor \frac{k}{2}\right\rfloor-\left\lfloor \frac{k+1}{2}\right\rfloor\right) x_1^{k-2}
+\left(\left\lfloor \frac{k}{2}\right\rfloor+\left\lfloor \frac{k(k-1)}{6}\right\rfloor
-\left\lfloor \frac{(k+1)k}{6}\right\rfloor\right) x_1^{k-4}+\cdots.
\end{align*}
\noindent 
When $k$ is even, the coefficient of $x_1^{k-2}$ is $1$, 
which checks the former part of $(\mathrm{i})$ already proved. 
When $k$ is odd, the coefficient of $x_1^{k-2}$ is $0$. However, it is easy to check that 
the coefficient of $x_1^{k-4}$ is given by $\left\lfloor \displaystyle\frac{k}{6}\right\rfloor$. 
If  $k \ge 7$, this gives the leading term. 
In the cases where $k=1,3,5$, 
we have $\mathrm{ch}(\mathfrak{h}_{1,1}(k))=0$ since the non-negative degree part of 
the character $\mathrm{ch}(\mathfrak{h}_{1,1}(k))$ vanishes and 
$\mathrm{ch}(\mathfrak{h}_{1,1}(k))$ is invariant under 
the involution $x_1 \leftrightarrow x_1^{-1}$. 

Finally we prove $\mathrm{(ii)}$. It is easy to see that the only isotypical component 
appearing in $\mathfrak{h}_{g,1}(k)_{h\leq 2}$ which has multiple double floors is
the one corresponding to the Young diagram $[k+1,k+1]$. Then the result follows from
the definition of $H_{[k+1]}$ given in Theorem \ref{thm:ortho}. 
\end{proof}

By combining Proposition \ref{prop:kl} with the latter part of Proposition \ref{prop:g1} $\mathrm{(i)}$, 
we obtain the following.

\begin{cor}
The leading term of the $\mathrm{GL}$-irreducible decomposition of $\mathfrak{h}_{g,1}(2k+1)$
is 
$$
\lfloor k/6\rfloor\,[2k, 3]_{\mathrm{GL}}\quad (k\geq 3).
$$ 
\end{cor}

The summands $[2k-2]_{\mathrm{Sp}}\subset \mathfrak{h}_{1,1}(2k)\ (k=1,2,\ldots)$ play a
central role in the theory of universal mixed elliptic motives due to Hain and Matsumoto
(cf. \cite{hain14} and \cite{pollack}).
More precisely, they consider the element
$$
\text{$\epsilon_{2k} \in [2k-2]_{\mathrm{Sp}}\subset \mathfrak{h}_{1,1}(2k)\ $: the highest weight vector } .
$$
An explicit description of this element was first given by Tsunogai \cite{tsunogai} 
(see also \cite{nakamura99}) and
twice of it is represented by the following Lie spider
\bigskip
\begin{center}
\includegraphics[width=.45\textwidth]{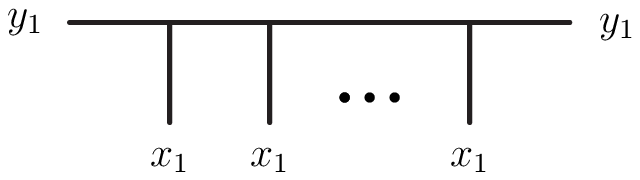}
\end{center}
On the other hand, the highest weight vector of $[2k, 2]_{\mathrm{GL}}\subset \mathfrak{h}_{g,1}(2k)$
is represented by the same Lie spider but we replace both $y_1$ by $x_2$.

They define
$$
\mathfrak{u}=\text{Lie subalgebra of $\mathfrak{h}_{1,1}$ generated by $\epsilon_{2k}$
for all $k=0,1,2,\ldots$}
$$
where $\epsilon_0\in\mathfrak{h}_{1,1}(0)\cong [2]_{\mathrm{Sp}}$ is the highest weight vector,
and they relate the structure of it with the theory of elliptic modular forms. Furthermore, they
consider the action of the absolute Galois group and, in particular, show the
existence of a certain homomorphism
$$
\mathfrak{f}\rightarrow \mathfrak{h}_{1,1}^{\mathrm{Sp}}
$$
from the ``fundamental Lie algebra"
$$
\mathfrak{f}=\text{free graded Lie algebra generated by $\sigma_3,\sigma_5,\ldots$}
$$
to the $\mathrm{Sp}$-invariant part of $\mathfrak{h}_{1,1}$. It has the property that
the image of $\sigma_{2k+1}$, denoted by $\tilde{\sigma}_{2k+1}$, lies in $\mathfrak{h}_{1,1}(4k+2)^{\mathrm{Sp}}$ and
moreover it should normalize $\mathfrak{u}$.


\begin{prop}
The kernel of the Enomoto-Satoh map $ES: \mathfrak{h}_{g,1}\rightarrow \mathfrak{a}_g$
is an $\mathrm{Sp}$-Lie subalgebra of $\mathfrak{h}_{g,1}$. 
\label{prop:sub}
\end{prop}
\begin{proof}
Since $ES$ is an $\mathrm{Sp}$-equivariant mapping, 
$\mathrm{Ker} ES$ is an $\mathrm{Sp}$-submodule of $\mathfrak{h}_{g,1}$.
On the other hand, a theorem of Satoh \cite{satoh} implies that $\mathrm{Ker}\, ES$ is a Lie subalgebra
of $\mathfrak{h}_{g,1}$ in a certain stable range, namely for a sufficiently large $g$.
Furthermore, ES is an i-stable homomorphism in the sense that the following
diagram is commutative: 
$$
\begin{CD}
\mathfrak{h}_{g,1}(k)  @>{ES}>> \mathfrak{a}_g(k-2)\\
@V{i}VV @VV{i}V\\
\mathfrak{h}_{g+1,1}(k)  @>{ES}>> \mathfrak{a}_{g+1}(k-2).
\end{CD}
$$
It follows that $\mathrm{Ker} ES$ is an $\mathrm{Sp}$-Lie subalgebra of $\mathfrak{h}_{g,1}$
without any condition on $g$. 
\end{proof}

\begin{thm}
$\mathrm{(i)}$ The summand $[2k-2]_{\mathrm{Sp}}\subset \mathfrak{h}_{1,1}(2k)$
lies in the kernel of the Enomoto-Satoh map
for any $k$. It follows that $ES(\mathfrak{u})=0.$

$\mathrm{(ii)}\ \text{The image of $\sigma_5$ in $\mathfrak{h}_{1,1}(10)^{\mathrm{Sp}}$ is characterized by the
condition that $ES([\epsilon_4,\tilde{\sigma}_5])=0$}$. More precisely, we have
$$
\mathrm{dim}\, \mathrm{Ker}\left(\mathfrak{h}_{1,1}(10)^{\mathrm{Sp}}
\xrightarrow{[\epsilon_4,\ ]}  \mathfrak{h}_{1,1}(14)
\xrightarrow{ES} \mathfrak{a}_1(12)\right)=1
$$
whereas $\mathrm{dim}\, \mathfrak{h}_{1,1}(10)^{\mathrm{Sp}}=3$.
\label{thm:z5}
\end{thm}

\begin{proof}
First we prove $\mathrm{(i)}$.
In view of Proposition \ref{prop:sub}, it suffices to show that $ES(\epsilon_{2k})=0$
for any $k\geq 1$ (the case $k=0$ is trivial), where we regard $\epsilon_{2k}$ as 
an element of $\mathfrak{h}_{1,1} (2k)\subset \mathfrak{h}_{g,1}(2k)$ after stabilizations. 
The Enomoto-Satoh map $ES$ is defined as the composition of group homomorphisms 
\begin{align*}
\mathfrak{h}_{g,1} (2k) &\subset H_\mathbb{Q} \otimes \mathcal{L}_{g,1} (2k+1) 
\subset H_\mathbb{Q}^{\otimes (2k+2)}\\
&\xrightarrow{K_{12}}  H_\mathbb{Q}^{\otimes 2k} 
\xrightarrow{proj} (H_\mathbb{Q}^{\otimes 2k})_{\mathbb{Z}/2k} \cong 
(H_\mathbb{Q}^{\otimes 2k})^{\mathbb{Z}/2k} = \mathfrak{a}_{g} (2k),
\end{align*}
\noindent
where $proj$ is the natural projection and 
the isomorphism $(H_\mathbb{Q}^{\otimes 2k})_{\mathbb{Z}/2k} \xrightarrow{\cong} 
(H_\mathbb{Q}^{\otimes 2k})^{\mathbb{Z}/2k}$ is given by 
\[u_1 \otimes u_2 \otimes \cdots \otimes u_{2k} \longmapsto
\sum_{i=1}^{2k} \sigma_{2k}^i (u_1 \otimes u_2 \otimes \cdots \otimes u_{2k})\]
using the cyclic permutation $\sigma_{2k}$ of entries of tensors. 
For $l \ge 1$, put 
$$
X_l=[x_1, [x_1, \ldots, [x_1,y_1]\cdots] =
(-1)^l [\cdots [y_1,x_1],x_1],\ldots,x_1] \in \mathcal{L}_{g,1} (l+1) \subset 
H_\mathbb{Q}^{\otimes (l+1)}.
$$ 

As an element of $H_\mathbb{Q} \otimes \mathcal{L}_{g,1} (2k+1) \subset 
H_\mathbb{Q}^{\otimes (2k+2)}$ , 
we have 
\begin{align*}
\epsilon_{2k}&=y_1 \otimes X_{2k} + x_1 \otimes [X_{2k-1},y_1] + 
x_1 \otimes [X_{2k-2}, -X_1] + x_1 \otimes [X_{2k-3}, (-1)^2 X_2]\\ &\quad 
+ \cdots +
x_1 \otimes [X_1, (-1)^{2k-2} X_{2k-2}] +x_1 \otimes [y_1, (-1)^{2k-1} X_{2k-1}] +y_1 \otimes ((-1)^{2k} X_{2k})\\
&= 2y_1 \otimes X_{2k} + 2x_1 \otimes [X_{2k-1},y_1] +
\sum_{i=1}^{2k-2} x_1 \otimes [X_i, (-1)^{i+1} X_{2k-1-i}]\\
&=2y_1 \otimes X_{2k} + 2(x_1 \otimes X_{2k-1} \otimes y_1-x_1 \otimes y_1 \otimes X_{2k-1}) \\
& \quad +
(-1)^{i+1}\sum_{i=1}^{2k-2} x_1 \otimes X_i \otimes X_{2k-1-i}
+(-1)^i\sum_{i=1}^{2k-2} x_1 \otimes  X_{2k-1-i} \otimes X_i.
\end{align*} 

Note that the contraction $K_{12}$ satisfies $K_{12} (X \otimes Y)=K_{12} (X) \otimes Y$ for 
any $X \in H_\mathbb{Q}^{\otimes i}$ and $Y \in H_\mathbb{Q}^{\otimes j}$ with $i \ge 2$.  
Hence we have 
\begin{align*}
K_{12}(\epsilon_{2k})&=
2K_{12}(y_1 \otimes X_{2k}) + 2\left\{K_{12}(x_1 \otimes X_{2k-1}) \otimes y_1
-K_{12}(x_1 \otimes y_1) \otimes X_{2k-1}\right\} \\
& \quad +
(-1)^{i+1}\sum_{i=1}^{2k-2} K_{12}(x_1 \otimes X_i) \otimes X_{2k-1-i}
+(-1)^i\sum_{i=1}^{2k-2} K_{12}(x_1 \otimes  X_{2k-1-i}) \otimes X_i.
\end{align*}

By induction, we can show that the equalities 
\begin{align*}
K_{12} (y_1 \otimes X_l) &= \sum_{i=1}^{l-1} (-1)^i X_{l-i} \otimes x_1^{\otimes (i-1)} +
(-1)^l y_1 \otimes   x_1^{\otimes (l-1)},\\
K_{12} (x_1 \otimes X_l) &= (-1)^l x_1^{\otimes l}
\end{align*}
\noindent
hold. Since 
\begin{align*}
x_1^{\otimes i} \otimes X_1 \otimes x_1^{\otimes (2k-2-i)}&=
x_1^{\otimes i} \otimes x_1 \otimes y_1 \otimes x_1^{\otimes (2k-2-i)}
-x_1^{\otimes i} \otimes y_1 \otimes x_1 \otimes x_1^{\otimes (2k-2-i)}\\
&=0\in (H_\mathbb{Q}^{\otimes 2k})_{\mathbb{Z}/2k},
\end{align*}
\noindent
we have 
$$
x_1^{\otimes j} \otimes X_{2k-1-j} =X_{2k-1-j} \otimes x_1^{\otimes j}= 0 
\in (H_\mathbb{Q}^{\otimes 2k})_{\mathbb{Z}/2k}.
$$
Therefore 
\[proj \circ {K}_{12} (\epsilon_{2k}) =
2 (-1)^{2k}y_1 \otimes x_1^{\otimes (2k-1)} 
+ 2(-1)^{2k-1} x_1^{\otimes (2k-1)}\otimes y_1=0.\]
This shows that $ES(\epsilon_{2k})=0$.

Next we prove $\mathrm{(ii)}$. 
We consider
the following $(3+3)$ linear chord diagrams 
\begin{align*}
D_1&=(12)(34)(56)(78)(9\ 10)(11\ 12),\ U_1=(12)(35)(46)(79)(8\ 10)(11\ 12),\\
D_2&=(12)(34)(56)(79)(8\ 11)(10\ 12),\ U_2=(12)(35)(47)(69)(8\ 10)(11\ 12),\\
D_3&=(12)(34)(56)(7\ 10)(8\ 11)(9\ 12),\ U_3=(16)(29)(38)(4\ 11)(5\ 10)(7\ 12)
\end{align*}
in $\mathcal{D}^\ell(12)$. We define
$$
u_i=\Phi(S'_{12} \circ \sigma_{12} \circ p'_{11} \circ \sigma_{12}^{-1}
(U_i))\in \mathfrak{h}_{1,1}(10)^{\mathrm{Sp}}\quad (i=1,2,3).
$$
It turns out that the set $\{D_1,D_2,D_3\}$ can serve as a detector of 
$\mathfrak{h}_{1,1}(10)^\mathrm{Sp}\cong\Q^3$. In fact, the intersection matrix
$(\alpha_{D_i}(u_j))$ is given by
$$
\begin{pmatrix}
46656 & 23328 & 3888\\
3456 & 192 & -576 \\
-27648 & -14304 & -4824
\end{pmatrix}
$$
which is non-singular. It follows that $\{u_1,u_2,u_3\}$ is a basis of $\mathfrak{h}_{1,1}(10)^\mathrm{Sp}$.
Next we compute the bracket $[\epsilon_4, u_i]\in\mathfrak{h}_{1,1}(14)$ and apply the Enomoto-Satoh
mapping to it. 
Then we obtain $3$ large tensors
$$
r_i=ES([\epsilon_4, u_i])\in\mathfrak{a}_1(12) \quad (i=1,2,3).
$$
Finally, we seek for a linear relation between $r_1,r_2,r_3$. It turns out that
there exists a unique relation
$$
41 r_1-51 r_2+4 r_3=0.
$$
We can now conclude that the element $41 u_1-51 u_2+4 u_3\in \mathfrak{h}_{1,1}(10)^{\mathrm{Sp}}$
is the unique element (up to scalars) such that its bracket with $\epsilon_4$
is contained in $\mathfrak{u}$,
completing the proof.






\end{proof}

\begin{remark}
Pollack \cite{pollack} determined the element $\tilde{\sigma}_5\in \mathfrak{h}_{1,1}(10)^{\mathrm{Sp}}\cong\Q^3$
explicitly. We have checked that the unique element in $\mathfrak{h}_{1,1}(10)^{\mathrm{Sp}}$ (up to scalars) given above,
coincides with his element. We are trying to extend our result to identify the element 
$\tilde{\sigma}_7\in \mathfrak{h}_{1,1}(14)^{\mathrm{Sp}}\cong H_{[8]}\cong\Q^{11}$.
\end{remark}
\section{Tables of orthogonal decompositions of $\mathfrak{h}_{g,1}(2k)^{\mathrm{Sp}}$}\label{sec:tables}

In this section, we give Tables for the orthogonal decompositions of $\mathfrak{h}_{g,1}(2k)^{\mathrm{Sp}}$
for the cases $2k=14,16,18, 20$.

\begin{table}[h]
\caption{$\text{Orthogonal decomposition of $\mathfrak{h}_{g,1}(14)^{\mathrm{Sp}}$}$}
\begin{center}
\begin{tabular}{|r|r|}
\noalign{\hrule height0.8pt}
\hfil $\mathrm{dim}\hspace{7mm}$& $\text{eigenspaces}\hspace{3.3cm}$   \\
\hline
$11\ (g=1)$ & $11[8]^\delta$  \\
\hline
$1691\ (g=2)$ & $147[71]^\delta 665[62]^\delta 752[53]^\delta 116[4^2]^\delta$ \\
\hline
$6471\ (g=3)$ & $403[61^2]^\delta 2044[521]^\delta 1436[431]^\delta 665[42^2]^\delta 232[3^22]^\delta$  \\
\hline
$8505\ (g=4)$ & $337[51^3]^\delta 1120[421^2]^\delta 266[3^21^2]^\delta 300[32^21]^\delta 11[2^4]^\delta$  \\
\hline
$8795\ (g=5)$ & $104[41^4]^\delta 168[321^3]^\delta 18[2^31^2]^\delta$  \\      
\hline
$8816\ (g=6)$ & $14[31^5]^\delta 7[2^21^4]^\delta $  \\ 
\hline
$8817\ (g\geq 7)$ & $[21^6]^\delta$  \\                                             
\noalign{\hrule height0.8pt}
\end{tabular}
\end{center}
\label{tab:h14}
\end{table}

\begin{table}[h]
\caption{$\text{Orthogonal decomposition of $\mathfrak{h}_{g,1}(16)^{\mathrm{Sp}}$}$}
\begin{center}
\begin{tabular}{|r|r|}
\noalign{\hrule height0.8pt}
\hfil $\mathrm{dim}\hspace{1cm}$& $\text{eigenspaces}\hspace{4.5cm}$   \\
\hline
$10\ (g=1)$ & $10[9]^\delta$  \\
\hline
$11842\ (g=2)$ & $440[81]^\delta 3028[72]^\delta 5860[63]^\delta 2504[54]^\delta$ \\
\hline
$69544\ (g=3)$ & $ 1776[71^2]^\delta 14616[621]^\delta  21204[531]^\delta 7664[52^2]^\delta $  \\
{} & $3270[4^21]^\delta  8904[432]^\delta 268[3^3]^\delta$  \\
\hline
$104190\ (g=4)$ & $ 2112[61^3]^\delta  12904[521^2]^\delta 9744[431^2]^\delta 6936[42^21]^\delta $  \\
{} & $2532[3^221]^\delta 418[32^3]^\delta$  \\
\hline
$110610\ (g=5)$ & $960[51^4]^\delta 3546[421^3]^\delta 823[3^21^3]^\delta 1059[32^21^2]^\delta 32[2^41]^\delta$  \\      
 \hline
$111131\ (g=6)$ & $180[41^5]^\delta 312[321^4]^\delta 29[2^31^3]^\delta$  \\ 
\hline
$111148\ (g\geq 7)$ & $12[31^6]^\delta 5[2^22^5]^\delta$  \\                                       
\noalign{\hrule height0.8pt}
\end{tabular}
\end{center}
\label{tab:h16}
\end{table}

\vspace{5cm}

\begin{table}[h]
\caption{$\text{Orthogonal decomposition of $\mathfrak{h}_{g,1}(18)^{\mathrm{Sp}}$}$}
\begin{center}
\begin{tabular}{|r|r|}
\noalign{\hrule height0.8pt}
\hfil $\mathrm{dim}\hspace{1cm}$& $\text{eigenspaces}\hspace{5cm}$   \\
\hline
$57\ (g=1)$ & $57[10]^\delta$  \\
\hline
$100908\ (g=2)$ & $1710[91]^\delta 15053[82]^\delta 42826[73]^\delta 36780[64]^\delta 4482[5^2]^\delta $ \\
\hline
$888099\ (g=3)$ & $ 8520[81^2]^\delta 97776[721]^\delta  239184[631]^\delta 78422[62^2]^\delta $  \\
{} & $ 117024[541]^\delta 191095[532]^\delta 36780 [4^22]^\delta
18390 [43^2]^\delta$   \\
\hline
$1548984\ (g=4)$ & $ 13584[71^3]^\delta  126540[621^2]^\delta  199320[531^2]^\delta 116340[52^21]^\delta 32676[4^21^2]^\delta$  \\
{}  & $145281[4321]^\delta 15053[42^3]^\delta 5919[3^31]^\delta 6172[3^22^2]^\delta$  \\      
\hline
$1710798\ (g=5)$ & $8842[61^5]^\delta 56280[521^3]^\delta 44151[431^3]^\delta 35220[42^21^2]^\delta$  \\ 
{}  & $13344[3^221^2]^\delta  3920[32^31]^\delta 57[2^5]^\delta$  \\   
\hline
$1728591\ (g=6)$ & $2600[51^5]^\delta 9619[421^4]^\delta 2340[3^21^4]  3096[32^21^3]^\delta 138[2^41^2]^\delta$  \\   
\hline
$1729620\ (g=7)$ & $357[41^6]^\delta 605[321^5]^\delta 67[2^31^4] $  \\  
\hline
$1729656\ (g=8)$ & $24[31^7]^\delta 12[2^21^6]^\delta $  \\  
\hline
$1729657\ (g\geq 9)$ & $[21^8]^\delta $  \\
\noalign{\hrule height0.8pt}
\end{tabular}
\end{center}
\label{tab:h18}
\end{table}

\begin{table}[h]
\caption{$\text{Orthogonal decomposition of $\mathfrak{h}_{g,1}(20)^{\mathrm{Sp}}$}$}
\begin{center}
\begin{tabular}{|r|r|}
\noalign{\hrule height0.8pt}
\hfil $\mathrm{dim}\hspace{1cm}$& $\text{eigenspaces}\hspace{5cm}$   \\
\hline
$108\ (g=1)$ & $108[11]^\delta$  \\
\hline
$869798\ (g=2)$ & $5815[10\, 1]^\delta 6829[92]^\delta  273592[83]^\delta 383950[74]^\delta 138042[65]^\delta $ \\
\hline
$12057806\ (g=3)$ & $ 37843[91^2]^\delta 596064[821]^\delta  2202900[731]^\delta 672107[72^2]^\delta 2157806[641]^\delta$  \\
{} & $2819712[632]^\delta 276504 [5^21]^\delta
1722706 [542]^\delta 537940[53^2]^\delta 164426[4^23]^\delta$   \\ 
\hline
$25062360\ (g=4)$ & $ 78887[81^3]^\delta  1065000[721^2]^\delta  2853514[631^2]^\delta 1491000[62^21]^\delta $  \\
{}  & $1468605[541^2]^\delta 3966480[5321]^\delta 319826[52^3]^\delta 792666[4^221]^\delta $  \\  
{}  & $485254[43^21]^\delta 452736[432^2]^\delta 30586[3^32]^\delta $  \\ 
\hline
$29129790\ (g=5)$ & $69015[71^4]^\delta 692160[621^3]^\delta 1136010[531^3]^\delta 739520[52^21^2]^\delta 189489[4^21^3]^\delta$  \\ 
{}  & $958680[4321^2]^\delta  165385[42^31]^\delta 41775[3^31^2]^\delta 69775[3^22^21]^\delta 5621[32^4]$  \\   
\hline
$29688027\ (g=6)$ & $28371[61^5]^\delta 188640[521^4]^\delta 150864[431^4]  125700[42^21^3]^\delta $  \\  
{} & $ 48300[3^221^3]^\delta 16060[32^31^2]^\delta 302[2^51]^\delta$\\
\hline
$29728348\ (g=7)$ & $5695[51^6]^\delta 21840[421^5]^\delta 5262[3^21^5]^\delta  7215[32^21^4]^\delta 309[2^41^3]^\delta$  \\  
\hline
$29729957\ (g=8)$ & $545[41^7]^\delta 968[321^6]^\delta 96[2^31^5]^\delta$  \\  
\hline
$29729988\ (g\geq9)$ & $21[31^8]^\delta 10[2^21^7]^\delta$  \\
\noalign{\hrule height0.8pt}
\end{tabular}
\end{center}
\label{tab:h20}
\end{table}

\vspace{5cm}

\newpage
\bibliographystyle{amsplain}

\end{document}